\documentclass[reqno]{amsart}
  \usepackage{mathabx}
\usepackage{cite}
\allowdisplaybreaks

 \date{\today}
\theoremstyle{plain}
\newtheorem{thm}{Theorem}[section]

\newtheorem{lem}[thm]{Lemma}
\newtheorem{prop}[thm]{Proposition}

\theoremstyle{definition}

\theoremstyle{remark}
\newtheorem{rem}{Remark}[section]
\numberwithin{equation}{section}
\renewcommand{\theequation}{\thesection.\arabic{equation}}

\renewcommand{\u}{{\bf u}}
\renewcommand{\H}{{\bf H}}
\renewcommand{\j}{{\bf J}}
\renewcommand{\i}{{\bf E}}
\renewcommand{\v}{{\vvvert}}
\newcommand{\dv}{{\rm div\,}}
\newcommand{\dif}{{\rm d}}
\newcommand{\cu}{{\rm curl\,}}

\newcommand{\U}{{\mathbf U}}
\newcommand{\F}{{\mathbf F}}
\newcommand{\G}{{\mathbf G}}
\newcommand{\W}{{\mathbf W}}
\newcommand{\D}{{\mathbf D}}
\newcommand{\A}{{\mathbf A}}
\newcommand{\s}{{\mathbf S}}
\newcommand{\R}{{\mathbf R}}
 
\newcommand{\p}{{\partial}}
\newcommand{\pa}{{\partial^\alpha_x}}

\begin{document}

\title[Complete electromagnetic fluid system to full MHD equations]
 {Convergence of the complete electromagnetic fluid system  to the full compressible
magnetohydrodynamic equations}
\author{Song Jiang}
\address{Institute of Applied Physics and Computational Mathematics, P.O. Box 8009, Beijing 100088, P.R. China}
 \email{jiang@iapcm.ac.cn}

\author [Fucai Li] {Fucai Li$\,^*$}%
 \thanks{$^*$Corresponding  author}
\address{Department  of Mathematics,  Nanjing University, Nanjing 210093, P.R. China}
 \email{fli@nju.edu.cn}

\begin{abstract}
The full compressible magnetohydrodynamic equations can be derived formally from the complete
electromagnetic  fluid  system in some sense as the dielectric constant tends to zero.
This process is usually referred as magnetohydrodynamic approximation  in physical books.
In this paper we justify this singular limit rigorously in the framework of smooth solutions
for  well-prepared  initial data.
  \end{abstract}



\maketitle

\renewcommand{\theequation}{\thesection.\arabic{equation}}
\setcounter{equation}{0}
\section{Introduction and Main Results} \label{S1}

Electromagnetic dynamics studies  the motion of an electrically conducting fluid in
 the presence of an electromagnetic field.
 In electromagnetic dynamics  the fluid and the electromagnetic field are connected closely with each other,
hence the fundamental system of electromagnetic dynamics usually contains
 the hydrodynamical equations and the electromagnetic ones.
The complete   electromagnetic  fluid  system includes the conservation of mass,
momentum, and energy to the fluid, the Maxwell system to the electromagnetic field, and the  conservation of electric charge, which take the forms (\!\cite{Im,K,EM})
\begin{align}
&\p_t \rho  +\dv(\rho\u)=0, \label{faa} \\
&\rho(\p_t\u +  \u\cdot \nabla\u)+\nabla P
  = \dv\Psi(\u)+\rho_\textrm{e}\i+\mu_0\j\times \H, \label{fab} \\
&\rho \frac{\p e}{\p \theta}(\p_t\theta+\u\cdot \nabla \theta)+\theta \frac{\p P}{\p \theta} \dv \u \nonumber\\
& \qquad \quad \quad =\dv(\kappa \nabla \theta)
 +\Psi(\u):\nabla\u+(\j-\rho_\textrm{e}\u)\cdot (\i+\mu_0\u\times \H) ,\label{fac}\\
&\epsilon\p_t\i-\cu \H+\j=0, \label{fad}\\
&\p_t \H+\frac{1}{\mu_0 }\cu \i=0, \label{fae}\\
&\p_t (\rho_\textrm{e})+\dv \j=0,\label{faf}\\
&\epsilon\dv \i=\rho_\textrm{e},\quad \dv \H=0.\label{fag}
\end{align}
The system \eqref{faa}--\eqref{fag} consists of $14$ equations in $12$ unknowns, namely, the mass density $\rho$,
the velocity $\u=(u_1,u_2,u_3)$, the absolute temperature $\theta$, the electric field $\i=(E_1,E_2,E_3) $,
the magnetic field $ \H=(H_1, H_2,H_3)$, and the electric charge density $\rho_\textrm{e}$.
The quantity $\Psi(\u)$ is the viscous stress tensor given by
\begin{equation}\label{psi}
\Psi(\u)=2\mu \mathbb{D}(\u)+\lambda\dv\u \;\mathbf{I}_3, \quad \mathbb{D}(\u)=(\nabla\u+\nabla\u^\top)/2,
\end{equation}
where $\mathbf{I}_3$  denotes the $3\times 3$ identity matrix,
and $\nabla \u^\top$   the transpose of the matrix
$\nabla \u$. The pressure $P=P(\rho,\theta)$ and
the internal energy $e=e(\rho,\theta)$ are smooth functions of $\rho$ and $\theta$ of the flow, and satisfy the  Gibbs relation
\begin{equation}\label{gibbs}
\theta \mathrm{d}S=\mathrm{d}e +
P\,\mathrm{d}\left(\frac{1}{\rho}\right),
\end{equation}
for some smooth function (entropy) $S=S(\rho,\theta)$ which expresses the first law of the
thermodynamics.
The current density $\j$ is expressed by Ohm's law, i.e.,
\begin{align}\label{ohm}
\j-\rho_\textrm{e} \u =\sigma (\i+\mu_0\u\times \H),
\end{align}
where $\rho_\textrm{e} \u$ is called the convection current.
 The symbol $\Psi(\u):\nabla\u$ denotes the scalar product of two matrices:
\begin{equation}\label{psiu}
\Psi(\u):\nabla\u=\sum^3_{i,j=1}\frac{\mu}{2}\left(\frac{\partial
u_i}{\partial x_j} +\frac{\partial u_j}{\partial
x_i}\right)^2+\lambda|\dv\u|^2=
2\mu|\mathbb{D}(\u)|^2+\lambda|\mbox{tr}\mathbb{D}(\u)|^2.
\end{equation}
 The viscosity coefficients $\mu$ and $\lambda $ of the fluid satisfy
 $\mu>0$ and $2\mu+3\lambda>0$. The parameters  $\epsilon>0$ is the dielectric constant,
$\mu_0 >0$ the magnetic permeability, $\kappa>0$ the heat
conductivity, and $\sigma>0$ the electric conductivity coefficient, respectively.
In general, the conductivity $\sigma$  may be a tensor depending on
the present magnetic field. However, in this paper, we
shall suppose that the Hall effect is negligible and $\sigma$ is a scalar quantity. If
the Hall effect is to be taken into account, \eqref{ohm}  must be replaced by
\begin{align*}
\j-\rho_\textrm{e} \u =\sigma (\i+\mu_0\u\times \H)-\frac{\sigma\sigma_0\mu_0}{n_\mathrm{e} e_0}(\i+\mu_0\u\times \H)\times \H,
\end{align*}
where $\sigma_0$ is the electric conductivity in the absence of a magnetic field, $n_\mathrm{e}$ the
electron number density, and $e_0$ the charge of an electron.
 For simplicity, in the following consideration, we shall assume that $\mu,\lambda, \epsilon,\mu_0, \kappa$, and $\sigma$ are constants.

Mathematically, it is very difficult to study the properties of solutions to the electromagnetic fluid
system \eqref{faa}--\eqref{fag}.
The reason is that, as pointed out by Kawashima \cite{K},  the system of the electromagnetic quantities
$(\mathbf{E},\mathbf{H},\rho_\textrm{e})$ in the system \eqref{faa}--\eqref{fag}, which is regarded as a first-order
hyperbolic system, is neither symmetric hyperbolic nor strictly
hyperbolic in the three-dimensional case.  The same difficulty also occurs in the
first-order hyperbolic system of $(\mathbf{E},\mathbf{H})$ which is
obtained from the above system by eliminating $\rho_\textrm{e}$ with the aid of the first equation
of \eqref{fag}. Therefore, the classic hyperbolic-parabolic theory (for example \cite{VH}) can not
be applied here.
There are only a few mathematical results on the electromagnetic fluid  system\eqref{faa}--\eqref{fag}
in some special cases. Kawashima~\cite{K} obtained the global existence of smooth solutions
in the two-dimensional case when the initial data are a small perturbation of some given constant state.
  Umeda, Kawashima and Shizuta~\cite{UKS} obtained the
global existence and time decay of smooth solutions to the
linearized equations of the system \eqref{faa}--\eqref{fag} in the three-dimensional case near some given
constant equilibria. Based on the above arguments, it is desirable to introduce some simplifications without sacrificing
the essential feature of the phenomenon.

As it was pointed out in \cite{Im}, the
assumption that the electric charge density $\rho_\textrm{e}\simeq 0$ is physically very reasonable
for the study of plasmas. In this situation,
the convection current $\rho_{\textrm{e}}\u $ is negligible in comparison with the conduction current
$\sigma (\i+\mu_0\u\times \H)$, thus
 we can eliminate
the terms involving $\rho_\textrm{e}$ in the
electromagnetic fluid  system \eqref{faa}--\eqref{fag} and obtain the following simplified system:
\begin{align}
&\p_t \rho  +\dv(\rho\u)=0, \label{faaa} \\
&\rho(\p_t\u +  \u\cdot \nabla\u)+\nabla P
  = \dv\Psi(\u)+\mu_0\j\times \H, \label{fabb} \\
&\rho \frac{\p e}{\p \theta}(\p_t\theta+\u\cdot \nabla \theta)+\theta \frac{\p P}{\p \theta} \dv \u =\kappa\Delta \theta
 +\Psi(\u):\nabla\u+\j\cdot (\i+\mu_0\u\times \H) ,\label{facc}\\
&\epsilon\p_t\i-\cu \H+\j=0, \label{fadd}\\
&\p_t \H+\frac{1}{\mu_0 }\cu \i=0, \quad \dv \H=0,  \label{faee}
\end{align}
with
\begin{align}\label{ohmm}
\j =\sigma (\i+\mu_0\u\times \H).
\end{align}

 We remark here that the assumption $\rho_\textrm{e}\simeq 0$ is quite different from
the assumption of exact neutrality $\rho_\textrm{e}=0$, which would lead to the superfluous
condition $\textrm{div} \i=0$ by  \eqref{fag}.


Formally, if we take the  dielectric constant $\epsilon =0$ in
\eqref{fadd}, i.e. the displacement current $\epsilon \partial_t \i$ is negligible, then we
obtain that $\j=\cu \H$. Thanks to \eqref{ohmm}, we can eliminate the
electric field $\i$ in \eqref{fabb}, \eqref{facc}, and \eqref{faee} and finally obtain the system
\begin{align}
&\p_t \rho  +\dv(\rho\u)=0, \label{fabba} \\
&\rho(\p_t\u +  \u\cdot \nabla\u)+\nabla P
  = \dv\Psi(\u)+\mu_0 \cu \H\times \H, \label{fabbb} \\
&\rho \frac{\p e}{\p \theta}(\p_t\theta+u\cdot \nabla \theta)+\theta \frac{\p P}{\p \theta} \dv \u =\kappa\Delta   \theta
 +\Psi(\u):\nabla\u+\frac{1}{\sigma }|\cu \H|^2, \label{fabbc}\\
&\partial_t \H -\cu(\u\times\H)= -\frac{1}{\sigma\mu_0}\cu (\cu\H ),\quad \dv\H=0.\label{fabbd}
\end{align}
The equations \eqref{fabba}--\eqref{fabbd}  are the so-called full
compressible magnetohydrodynamic equations, see \cite{EM,KL,LL}.
It should be pointed that although it has been completely eliminated in the limit equations
\eqref{fabba}--\eqref{fabbd}, the electric field $\i$ still plays an essentially important role in the phenomena
under consideration. In fact, it determines the electric current $\sigma (\i+\mu_0\u\times \H)$ which
generates the magnetic field $\H$. The electric field $\i$ and the magnetic field $\H$ satisfy the relation
$$\i=\frac{1}{\sigma}\cu \H-\mu_0\u\times \H.$$

The above formal derivation is usually referred as magnetohydrodynamic
approximation, see \cite{Im,EM}. In \cite{KS1}, Kawashima and
Shizuta justified this limit process rigorously  in the two-dimensional case  for local smooth solutions, i.e., $\u=(u_1,u_2,0)$, $\i=(0,0,E_3)$,
  and $\H=(H_1,
  $ $H_2,0)$ with spatial variable $x=(x_1,x_2)\in \mathbb{R}^2$. In this situation, we can obtain
  that $\rho_{\textrm{e}}=0$ and the system \eqref{faa}--\eqref{fag} is reduced to \eqref{faaa}--\eqref{ohmm}.
Later, in \cite{KS2}, they also obtained the global convergence of
the limit in the two-dimensional case under the assumption that both
the initial data of the electromagnetic fluid equations and those of the
compressible magnetohydrodynamic equations are a small perturbation
of some given constant state in some Sobolev spaces in which the global
smooth solution can be obtained. Recently, we studied the magnetohydrodynamic
approximation for the isentropic electromagnetic fluid system in a three-dimensional period domain and deduced the
isentropic compressible magnetohydrodynamic
equations \cite{JL}.


 The purpose of this paper is to give a  rigorous derivation of the full
  compressible magnetohydrodynamic equations \eqref{fabba}--\eqref{fabbd} from the
  electromagnetic fluid
  system \eqref{faaa}--\eqref{ohmm} as the dielectric constant $\epsilon$ tends to
  zero.
 For the sake of simplicity and clarity of presentation, we shall focus on the ionized fluids obeying the perfect gas relations
\begin{align} \label{hpg}
P=\mathfrak{R}\rho \theta,\quad e=c_V\theta,
\end{align}
where the parameters $\mathfrak{R}>0$ and $ c_V\!>\!0$ are the gas constant and
the heat capacity at constant volume, respectively.
 We consider the system \eqref{faaa}--\eqref{ohmm} in a periodic domain of $\mathbb{R}^3$, i.e., the torus
$\mathbb{T}^3=(\mathbb{R}/(2\pi \mathbb{Z}))^3$.


Below for simplicity of presentation, we take the physical constants
$\mathfrak{R}, c_V, \sigma$, and $\mu_0$ to be one.
To emphasize the unknowns depending  on the small parameter
$\epsilon$, we rewrite the electromagnetic fluid system
\eqref{faaa}--\eqref{ohmm} as
\begin{align}
 &\partial_t \rho^\epsilon  +\dv(\rho^\epsilon\u^\epsilon)=0, \label{fya} \\
&\rho^\epsilon(\p_t\u^\epsilon +  \u^\epsilon\cdot \nabla\u^\epsilon)+\nabla (\rho^\epsilon\theta^\epsilon)
  = \dv\Psi(\u^\epsilon)+(\i^\epsilon+\u^\epsilon\times \H^\epsilon)\times \H^\epsilon, \label{fyb} \\
&\rho^\epsilon  (\p_t\theta^\epsilon+\u^\epsilon\cdot \nabla \theta^\epsilon)+\rho^\epsilon\theta^\epsilon  \dv \u^\epsilon   =\kappa\Delta  \theta^\epsilon
 +\Psi(\u^\epsilon):\nabla\u^\epsilon +  |\i^\epsilon+\u^\epsilon\times \H^\epsilon|^2 ,\label{fyc}\\
&\epsilon \partial_t \i^\epsilon-\cu\H^\epsilon + \i^\epsilon+\u^\epsilon\times \H^\epsilon=0,\label{fyd}\\
&\partial_t\H^\epsilon +\cu \i^\epsilon=0, \quad \dv\H^\epsilon=0, \label{fye}
\end{align}
where  $\Psi(\u^\epsilon)$ and $\Psi(\u^\epsilon):\nabla\u^\epsilon$
are defined through \eqref{psi} and \eqref{psiu} with $\u$ replaced by $\u^\epsilon$.
The system \eqref{fya}--\eqref{fye} is supplemented with the initial data
\begin{align}\label{fyf}
 (\rho^\epsilon, \u^\epsilon, \theta^\epsilon, \i^\epsilon,\H^\epsilon)|_{t=0}
 =( \rho_0^\epsilon(x), \u_0^\epsilon(x),\theta_0^\epsilon(x), \i_0^\epsilon(x),\H_0^\epsilon(x)), \quad x\in \mathbb{T}^3.
\end{align}

We also rewrite the limiting equations \eqref{fabba}--\eqref{fabbd} (recall that $\mathfrak{R}=c_V=\sigma=\mu_0=1  $) as
\begin{align}
&\partial_t \rho^0  +\dv(\rho^0\u^0)=0, \label{fza} \\
&\rho^0(\p_t\u^0 +  \u^0\cdot \nabla\u^0)+\nabla (\rho^0\theta^0)
  = \dv\Psi(\u^0)+\cu \H^0\times \H^0, \label{fzb} \\
&\rho^0  (\p_t\theta^0+\u^0\cdot \nabla \theta^0)+\rho^0\theta^0  \dv \u^0   =\kappa\Delta  \theta^0
 +\Psi(\u^0):\nabla \u^0+  |\cu \H^0|^2 ,\label{fzc}\\
&\partial_t \H^0 - \cu(\u^0\times\H^0)=-\cu(\cu\H^0),\quad \dv\H^0=0, \label{fzd}
\end{align}
where  $\Psi(\u^0)$ and $\Psi(\u^0):\nabla\u^0$
are defined through \eqref{psi} and \eqref{psiu} with $\u$ replaced by $\u^0$.
The system  \eqref{fza}--\eqref{fzd} is  equipped with the  initial data
\begin{align}\label{fze}
 (\rho^0, \u^0,\theta^0, \H^0)|_{t=0}
 =( \rho^0_0(x), \u^0_0(x), \theta_0^0(x), \H_0^0(x)), \quad x\in \mathbb{T}^3.
\end{align}

Notice that  the  electric field $\i^0$ is induced according to the relation
\begin{equation}\label{Ohm}
\i^0=\cu\H^0 - \u^0\times\H^0
\end{equation}
by moving the conductive flow in the magnetic field.

Before stating our main results, we recall the local existence  of smooth solutions to the problem \eqref{fza}--\eqref{fze}.
Since the system \eqref{fza}--\eqref{fzd} is parabolic-hyperbolic, the results in \cite{VH} imply that
 \begin{prop}[\cite{VH}]\label{Pa} Let $s>7/2$ be an integer and
 assume that the initial data $(\rho^0_0, \u^0_0,\theta^0_0,  \H_0^0)$ satisfy
\begin{gather*}
 \rho^0_0, \u^0_0, \theta^0_0, \H_0^0\in H^{s+2}(\mathbb{T}^3), \ \ \dv \H^0_0 =0,\nonumber\\
    0<\bar \rho= \inf_{x\in \mathbb{T}^3}\rho^0_0(x)\leq \rho^0_0(x)\leq
    \bar {\bar \rho}= \sup_{x\in \mathbb{T}^3}\rho^0_0(x)<+\infty,\\
    0<\bar \theta= \inf_{x\in \mathbb{T}^3}\theta^0_0(x)\leq \theta^0_0(x)\leq
    \bar {\bar \theta}= \sup_{x\in \mathbb{T}^3}\theta^0_0(x)<+\infty
\end{gather*}
 for some positive constants $\bar \rho, \bar{\bar\rho}, \bar \theta$, and $\bar{\bar\theta}$. Then there exist positive
 constants $T_*\,($the maximal time interval, $ 0<T_*\leq +\infty )$, and $\hat \rho, \tilde{\rho}, \hat \theta, \tilde{\theta} $,
 such that the problem
\eqref{fza}--\eqref{fze} has a unique classical solution $(\rho^0,\u^0,\theta^0,\H^0)$ satisfying $\dv  \H^0=0$ and
\begin{gather*}
   \rho^0 \in C^l([0,T_*),H^{s+2-l}(\mathbb{T}^3)),\, \,\, \u^0, \theta^0,
\H^0 \in  C^l([0,T_*),H^{s+2-2l}(\mathbb{T}^3)), \ \   l=0,1;\ \
 \\
    0<\hat \rho= \inf_{(x,t)\in \mathbb{T}^3\times [0,T_*)}\rho^0(x,t)\leq  \rho^0(x,t)\leq
    {\tilde \rho}= \sup_{(x,t)\in \mathbb{T}^3\times [0,T_*)}\rho^0(x,t)<+\infty,\\
       0<\hat \theta= \inf_{(x,t)\in \mathbb{T}^3\times [0,T_*)}\theta^0(x,t)\leq \theta^0(x,t)\leq
    {\tilde \theta}= \sup_{(x,t)\in \mathbb{T}^3\times [0,T_*)}\theta^0(x,t)< +\infty.
\end{gather*}

 \end{prop}

\smallskip
The  main results of this paper can be stated as follows.
\begin{thm}\label{th}
Let $s> 7/2$ be an integer and $(\rho^0, \u^0,\theta^0, \H^0)$
be the unique classical solution to the problem \eqref{fza}--\eqref{fze} given in Proposition \ref{Pa}.
 Suppose that the initial data $(\rho^\epsilon_0, \u^\epsilon_0,\theta_0^\epsilon, \i_0^\epsilon,
\H_0^\epsilon)$ satisfy
$$
 \rho^\epsilon_0, \u^\epsilon_0,\theta_0^\epsilon,  \i^\epsilon_0, \H^\epsilon_0\in H^{s}(\mathbb{T}^3), \  \dv \H^\epsilon_0 =0,
 \  \inf_{x\in \mathbb{T}^3}
   \rho^\epsilon_0(x)>0, \      \inf_{x\in \mathbb{T}^3}
  \theta^\epsilon_0(x)>0,
 $$
and
\begin{align}
&  \Vert (\rho^\epsilon_0-\rho^0_0, \u^\epsilon_0-\u^0_0, \theta^\epsilon_0-\theta^0_0,
\H_0^\epsilon-\H_0^0) \Vert_{s}\nonumber\\
&  \qquad \quad
+  \sqrt{\epsilon}\left\Vert\i^\epsilon_0- ( \cu\H^0_0
- \u^0_0\times\H^0_0 )\right\Vert_{s} \leq  L_0 {\epsilon}, \label{ivda}
 \end{align}
for some constant $L_0>0$. Then, for any $T_0\in (0,T_* )$, there exist a constant $L>0$, and
a sufficient small constant $\epsilon_0>0$, such that for any $\epsilon\in
(0,\epsilon_0]$, the problem \eqref{fya}--\eqref{fyf} has a unique smooth solution $(\rho^\epsilon,
\u^\epsilon, \theta^\epsilon, \i^\epsilon,\H^\epsilon)$ on $[0,T_0]$ enjoying
\begin{align}\label{iivda}
  &  \Vert (\rho^\epsilon-\rho^0, \u^\epsilon-\u^0, \theta^\epsilon-\theta^0,\H^\epsilon-\H^0)(t)
\Vert_{s} \nonumber\\
&  \qquad \quad +  \sqrt{\epsilon}\left\Vert\left\{\i^\epsilon- ( \cu\H^0
- \u^0\times\H^0 )\right\}(t)\right\Vert_{s} \leq L {\epsilon},  \ \ t\in [0,T_0].
 \end{align}
 Here $\|\cdot\|_{s}$ denotes the norm of Sobolev space $H^s(\mathbb{T}^3)$.
\end{thm}

\begin{rem}
  The inequality \eqref{iivda} implies that the sequences $(\rho^\epsilon, \u^\epsilon,\theta^\epsilon, \H^\epsilon)$
  converge strongly to $(\rho^0,\u^0,\theta^0,\H^0)$ in $L^\infty(0,T; H^{s}(\mathbb{T}^3))$ and
  $\i^\epsilon$ converge strongly to $\i^0$ in $L^\infty(0,T; H^{s}(\mathbb{T}^3))$  but with different convergence rates, where
$\i^0$ is defined by \eqref{Ohm}.
\end{rem}

\begin{rem}
  Theorem \ref{th}  still holds for the case with general state equations with minor modifications.
  Furthermore, our results also hold in the whole space $\mathbb{R}^3$.
  Indeed,  neither the compactness of $\mathbb{T}^3$ nor   Poincar\'{e}-type inequality is used in our arguments.
\end{rem}

\begin{rem} In the two-dimensional case,
our result is similar to that of \cite{KS1} (see Remark 5.1 of \cite{KS1}).  In addition, if we assume that
the initial data  are a small perturbation of some given constant state
in the Sobolev norm $H^s(\mathbb{T}^3)$ for $s>3/2+2$, we can extend the local convergence
result stated in Theorem \ref{th} to a global one.
\end{rem}

\begin{rem}
  For the local existence of solutions $(\rho^0,\u^0,\theta^0,\H^0)$ to the problem \eqref{fza}--\eqref{fze}, the
  assumption on the regularity of initial data $(\rho^0_0,\u^0_0,\theta^0_0,\H^0_0)$ belongs to
  $H^s(\mathbb{T}^3)$, $s>7/2$, is enough. Here we have added more regularity assumption in Proposition \ref{Pa} to
  obtain more regular solutions which are needed in the proof of Theorem \ref{th}.
\end{rem}


\begin{rem}
The viscosity and heat conductivity terms
in the system \eqref{fya}--\eqref{fye} play a crucial role in our uniformly bounded  estimates
(in order to control some undesirable higher-order terms).
In the case  of $\lambda=\mu=\kappa=0$, the original system \eqref{fya}--\eqref{fye} are reduced to the so-call non-isentropic
Euler-Maxwell system. Our arguments can not be applied to this case directly, for more details, see \cite{JL2}.
\end{rem}

We give some comments on the proof of Theorem \ref{th}. The main difficulty in dealing with the zero dielectric constant
limit problem is the oscillatory behavior of the electric field
as pointed out in \cite{JL}, besides the singularity in the Maxwell equations,
there exists an extra singularity caused by the strong coupling of
the electromagnetic field (the nonlinear source term) in the
momentum equation.
 Moreover, comparing to the isentropic case studied in \cite{JL}, we have to circumvent additional difficulties in the derivation of uniform estimates
induced by the nonlinear differential terms  (such as $\Psi(\u^\epsilon):\nabla\u^\epsilon$) and  higher order nonlinear terms
 (such as $ |\i^\epsilon+\u^\epsilon\times \H^\epsilon|^2$) involving $\u^\epsilon, \i^\epsilon$, and $\H^\epsilon $ in the temperature equation.
 In this paper, we shall
 overcome all these difficulties
and derive rigorously the full compressible
magnetohydrodynamic equations from the electromagnetic fluid equations
  by adapting the elaborate nonlinear energy method developed in
\cite{PW,JL}.
First, we derive the error system \eqref{error1}--\eqref{error4} by  utilizing  the original system  \eqref{fya}--\eqref{fye}
and the limit equations \eqref{fza}--\eqref{fzd}. Next, we study the estimates of $H^s$-norm to the error system.
 To do so, we shall make full use of  the special structure of the error system,  Sobolev imbedding, the Moser-type
inequalities, and the regularity of limit equations.
In particular,   very refined analyses are carried out to
deal with the higher order nonlinear terms in the system \eqref{error1}--\eqref{error4}.
Finally, we combine these obtained estimates and apply
Gronwall's type inequality to get the desired results. We remark that in the isentropic case in \cite{JL},
the density is controlled by the pressure, while in our case the density is controlled through the
viscosity terms in the momentum equations.



It should be pointed out that there are a lot of works on the studies of
compressible  magnetohydrodynamic equations by physicists and mathematicians
due to its physical importance, complexity, rich phenomena, and mathematical
challenges. Below we just mention some mathematical results on the full
 compressible magnetohydrodynamic equations \eqref{fabba}--\eqref{fabbd}, we refer the interested reader to
\cite{C,KL,LL,PD} for many discussions on  physical aspects.
For the one-dimensional planar compressible magnetohydrodynamic equations, the  existence
of global smooth solutions  with small initial data was shown in \cite{KO}. In \cite{TH,HT},
Hoff and Tsyganov obtained the global existence and uniqueness of weak solutions with small initial energy.
Under some technical conditions on the heat conductivity coefficient, Chen and Wang  \cite{CWa,CWb,W} obtained the existence,
uniqueness, and Lipschitz continuous dependence of global strong solutions with large initial data, see also
\cite{FJNa,FJNb} on the global existence and uniqueness of global weak solutions,
and \cite{FHL} on the global existence and uniqueness of large strong solutions with large initial data and vaccum.
For the full multi-dimensional compressible magnetohydrodynamic equations, the existence of variational
solutions was established in \cite{DF,FYa,HW2}, while a unique local strong solution was obtained in \cite{FYb}.
The low Mach number limit is a very interesting topic in magnetohydrodynamics, see \cite{K10,JLL,KT,NRT}
in the framework of the so-called variational solutions, and \cite{JJL3,JJL4,JJLX} in the framework of the local
smooth solutions with small density and temperature variations, or large density/entropy and temperature variations.

\medskip
Before ending this introduction, we give some notations and recall some basic facts which
will be frequently used throughout this paper.

(1) We denote by $\langle \cdot,\cdot\rangle$ the standard inner product in $L^2(\mathbb{T}^3)$
with $\langle f,f\rangle=\|f\|^2$, by
$H^k$ the standard Sobolev space $W^{k,2}$ with norm $\|\cdot\|_{k}$.  The notation $\|(A_1,A_2, \dots,
 A_l)\|_k$ means the summation of $\|A_i\|_k$ from $i=1$ to $i=l$.
For a multi-index $\alpha = (\alpha_1,  \alpha_2, \alpha_3)$,  we denote
$\partial_x^\alpha =\partial^{\alpha_1}_{x_1}\partial^{\alpha_2}_{x_2}\partial^{\alpha_3}_{x_3}$ and
$|\alpha|=|\alpha_1|+|\alpha_2|+|\alpha_3|$. For an integer $m$, the symbol $D^m_x$ denotes
the summation of all terms $\partial_x^\alpha$ with the multi-index $\alpha$ satisfying $|\alpha|=m$. We use $C_i$,
$\delta_i$, $K_i$, and $K$ to denote the constants which are independent of
$\epsilon$ and may change from line to line. We also omit the  spatial domain $\mathbb{T}^3$
in integrals for convenience.

(2) We shall frequently use the following Moser-type calculus
inequalities (see \cite{KM1}):

\hskip 4mm (i)\ \ For $f,g\in H^s(\mathbb{T}^3)\cap L^\infty(\mathbb{T}^3)$ and $|\alpha|\leq
s$, $s>3/2$, it holds that
\begin{align}\label{ma}
\|\partial^\alpha_x(fg)\| \leq C_s(\|f\|_{L^\infty}\|D^s_x
g\| +\|g\|_{L^\infty}\|D^s_x f\|).
\end{align}

\hskip 4mm (ii)\ \ For $f\in H^s(\mathbb{T}^3), D_x^1 f\in L^\infty(\mathbb{T}^3), g\in H^{s-1}(\mathbb{T}^3)\cap
L^\infty(\mathbb{T}^3)$ and $|\alpha|\leq s$, $s>5/2$, it holds that
\begin{align}\label{mb}
\quad  \ \ \|\partial^\alpha_x(fg)-f \partial^\alpha_xg\|\leq
C_s(\|D^1_x f\|_{L^\infty}\|D^{s-1}_x g\| +\|g\|_{L^\infty}\|D^s_xf\|). 
\end{align}

(3) Let $s> 3/2$, $f\in C^s(\mathbb{T}^3)$, and  $u\in H^s(\mathbb{T}^3)$, then for each multi-index $\alpha$, $1\leq |\alpha| \leq s$, we have
(\cite{Mo,KM1}):
\begin{align}\label{mo}
   \|\partial^\alpha_x (f(u))\| \leq C(1+\|u\|_{L^\infty}^{|\alpha|-1})\|u\|_{|\alpha|};
\end{align}
moreover, if $f(0)=0$, then (\cite{Ho97})
\begin{align}\label{ho}
 \|\partial^\alpha_x (f(u))\| \leq C( \|u\|_s)\|u\|_s.
\end{align}

This paper is organized as follows. In Section \ref{S2}, we utilize the primitive system \eqref{fya}--\eqref{fye} and the
target system \eqref{fza}--\eqref{fzd} to  derive the  error
system and state the local existence of the solution.
  In Section \ref{S3} we give the a priori energy
estimates of the error system  and present  the proof of Theorem \ref{th}.

\renewcommand{\theequation}{\thesection.\arabic{equation}}
\setcounter{equation}{0}
\section{Derivation of the  error system and  local existence} \label{S2}

In this section  we first derive the error system from the original
system \eqref{fya}--\eqref{fye} and the limiting equations
\eqref{fza}--\eqref{fzd}, then we state the local existence of solution to
this error system.

Setting $N^\epsilon=\rho^\epsilon-\rho^0,   \U^\epsilon=\u^\epsilon-\u^0, \Theta^\epsilon=\theta^\epsilon-\theta^0,
\F^\epsilon=\i^\epsilon-\i^0$, and $ \G^\epsilon=\H^\epsilon-\H^0$, and utilizing
      the system
\eqref{fya}--\eqref{fye} and the system \eqref{fza}--\eqref{fzd} with \eqref{Ohm}, we
  obtain that
 \begin{align}
 &  \partial_t N^\epsilon  +(N^\epsilon+\rho^0)\dv \U^\epsilon+(\U^\epsilon+\u^0)\cdot\nabla N^\epsilon
 =-N^\epsilon \dv \u^0-\nabla \rho^0\cdot \U^\epsilon, \label{error1} \\
 & \partial_t\U^\epsilon  +[(\U^\epsilon+\u^0)\cdot \nabla]\U^\epsilon
 +\nabla \Theta^\epsilon+\frac{\Theta^\epsilon+\theta^0}{N^\epsilon+\rho^0}\nabla N^\epsilon-\frac{1}{N^\epsilon+\rho^0}\dv\Psi(\U^\epsilon)\nonumber\\
  & \quad\qquad   =-(\U^\epsilon \cdot \nabla )\u^0-
                  \left[\frac{\Theta^\epsilon+\theta^0}{N^\epsilon+\rho^0}-\frac{\theta^0}{\rho^0}\right]\nabla  \rho^0
                       +\left[\frac{1}{N^\epsilon+\rho^0}-\frac{1}{\rho^0}\right]\dv\Psi(\u^0)\nonumber\\
 &\qquad \qquad               -\frac{1}{  {\rho^0}} \cu \H^0\times \H^0
              +\frac{1}{N^\epsilon+\rho^0}[\F^\epsilon+\u^0\times \G^\epsilon+\U^\epsilon\times \H^0]\times \H^0 \nonumber\\
  &\qquad \qquad +\frac{1}{N^\epsilon+\rho^0}[\F^\epsilon+\u^0\times \G^\epsilon+\U^\epsilon\times\H^0]\times \G^\epsilon\nonumber\\
 &\qquad \qquad + \frac{1}{N^\epsilon+\rho^0}(\U^\epsilon\times \G^\epsilon)\times (\G^\epsilon+\H^0), \label{error2}\\
 & \partial_t\Theta^\epsilon  +[(\U^\epsilon+\u^0)\cdot \nabla]\Theta^\epsilon
 + (\Theta^\epsilon+\theta^0)\, \dv \U^\epsilon-\frac{\kappa}{N^\epsilon+\rho^0}\Delta \Theta^\epsilon\nonumber\\
  & \quad\qquad   =-(\U^\epsilon \cdot \nabla )\theta^0  -\Theta^\epsilon \dv \u^0
    +\left[\frac{\kappa}{N^\epsilon+\rho^0}-\frac{\kappa}{\rho^0}\right]\Delta \theta^0    \nonumber\\
    &\qquad \qquad + \frac{ 2\mu}{N^\epsilon+\rho^0} |\mathbb{D}(\U^\epsilon)|^2
    +\frac{\lambda}{N^\epsilon+\rho^0}|\mbox{tr}\mathbb{D}(\U^\epsilon)|^2\nonumber\\
      &\qquad \qquad
   + \frac{ 4\mu}{N^\epsilon+\rho^0}\mathbb{D}(\U^\epsilon): \mathbb{D}(\u^0)
     + \frac{ 2\lambda}{N^\epsilon+\rho^0}\,[\mbox{tr}\mathbb{D}(\U^\epsilon) \mbox{tr}\mathbb{D}(\u^0)]\nonumber\\
      &\qquad \qquad + \left[\frac{2\mu}{N^\epsilon+\rho^0}-\frac{2\mu}{\rho^0}\right]|\mathbb{D}(\u^0)|^2
      +\left[\frac{\lambda}{N^\epsilon+\rho^0}-\frac{\lambda}{\rho^0}\right](\mbox{tr}\mathbb{D}(\u^0))^2\nonumber\\
    &\qquad \qquad + \frac{1}{N^\epsilon+\rho^0}|\F^\epsilon+\U^\epsilon\times \G^\epsilon|^2
    + \frac{1}{N^\epsilon+\rho^0} |\u^0\times \G^\epsilon+\U^\epsilon\times \H^0|^2\nonumber\\
    &\qquad \qquad + \frac{2}{N^\epsilon+\rho^0}(\F^\epsilon+\U^\epsilon\times \G^\epsilon)\cdot
                   [\cu \H^0+\u^0\times \G^\epsilon+\U^\epsilon\times \H^0]\nonumber\\
      &\qquad \qquad +   \frac{2}{N^\epsilon+\rho^0}\cu \H^0\cdot (\u^0\times \G^\epsilon+\U^\epsilon\times \H^0)\nonumber\\
        &\qquad \qquad  +\left[\frac{1}{N^\epsilon+\rho^0}-\frac{1}{\rho^0}\right]|\cu \H^0|^2,      \label{error22}\\
  & \epsilon \partial_t \F^\epsilon - \cu \G^\epsilon
 = -  [\F^\epsilon+\U^\epsilon\times \H^0+\u^0\times \G^\epsilon]-  \U^\epsilon\times \G^\epsilon\nonumber\\
 &\qquad \qquad\qquad\,\,\,\,\, \quad - {\epsilon} \partial_t \cu \H^0+\epsilon\partial_t(\u^0\times \H^0), \label{error3}\\
 & \partial_t \G^\epsilon+\cu \F^\epsilon =0,\quad  \dv \G^\epsilon=0,  \label{error4}
 \end{align}
 with initial data
 \begin{align}\label{error5}
 &  ( N^\epsilon,\U^\epsilon,\Theta^\epsilon,\F^\epsilon,\G^\epsilon)|_{t=0}:=
 ( N^\epsilon_0,\U^\epsilon_0,\Theta^\epsilon_0,\F^\epsilon_0,\G^\epsilon_0)\nonumber\\
  &  \quad  \quad= \left( \rho^\epsilon_0-\rho^0_0, \u^\epsilon_0-\u^0_0, \theta^\epsilon_0-\theta^0_0,
   \i^\epsilon_0- ( \cu\H^0_0 - \u^0_0\times\H^0_0 ),\H^\epsilon_0-\H^0_0 \right).
 \end{align}

 Denote
 \begin{align*}
 &\W^\epsilon=\left(\begin{array}{c}
                   N^\epsilon \\
                    \U^\epsilon\\
                    \Theta^\epsilon\\
                   \F^\epsilon \\
                    \G^\epsilon
                  \end{array}\right),
                  \ \
       \W^\epsilon_0=\left(\begin{array}{c}
                     N^\epsilon_0 \\
                    \U^\epsilon_0\\
                     \Theta^\epsilon_0\\
                     \F^\epsilon_0\\
                     \G^\epsilon_0\\
                  \end{array}\right), \ \
   \D^\epsilon=\left(\begin{array}{cc}
                   \mathbf{I}_{5} & \mathbf{0} \\
                   \mathbf{0} & \left(\begin{array}{cc}
                    \epsilon   \mathbf{I}_{3} & \mathbf{0}\\
                    \mathbf{0} &  \mathbf{I}_{3}
                    \end{array}
                    \right)
                  \end{array}\right), \\
 &   \A^\epsilon_i=\left(\begin{array}{cc}
                  \left(\begin{array}{ccc}
                   (\U^\epsilon+\u^0)_i & (N^\epsilon+\rho^0) e^\mathrm{T}_i & 0  \\
                    \frac{\Theta^\epsilon+\theta^0}{N^\epsilon+\rho^0}e_i & (\U^\epsilon+\u^0)_i \mathbf{I}_{3} & e_i\\
                    0 & (\Theta^\epsilon+\theta^0)e^\mathrm{T}_i & (\U^\epsilon+\u^0)_i
                    \end{array}\right)   & \mathbf{0} \\
                    \mathbf{0} &  \left(\begin{array}{cc}
                      \mathbf{0}& B_{i} \\
                    B_{i}^\mathrm{T}  & \mathbf{0}
                    \end{array}
                    \right)
                  \end{array}\right), \\
  &         \A^\epsilon_{ij} =  \left(\begin{array}{cccc}
                                    {0} & \mathbf{0}& \mathbf{0}& \mathbf{0}\\
                              \mathbf{0}&   \frac{\mu}{N^\epsilon+\rho^0}( e_ie_j^\mathrm{T}\mathbf{I}_{3}+e_i^\mathrm{T}e_j)
                                +  \frac{\lambda}{N^\epsilon+\rho^0}e^{\mathrm{T}}_je_i &\mathbf{0}& \mathbf{0} \\
                                     \mathbf{0} &   \mathbf{0}&\frac{\kappa}{N^\epsilon+\rho^0} e_ie_j^\mathrm{T}&   \mathbf{0}\\
                     \mathbf{0}  &   \mathbf{0} &   \mathbf{0}&  \mathbf{0}\\
                    \end{array}
                    \right),\\
  &\s^\epsilon(\W^\epsilon)=\left(\begin{array}{c}
                   -N^\epsilon \dv \u^0-\nabla \rho^0\cdot \U^\epsilon\\
                     {\R}^\epsilon_1\\
                    \R^\epsilon_2\\
                    \R^\epsilon_3\\
                    \mathbf{0}
                    \end{array}
                    \right),
 \end{align*}
where $\R^\epsilon_1, \R^\epsilon_2$, and $ \R^\epsilon_3$ denote the right-hand side of
\eqref{error2}, \eqref{error22}, and \eqref{error3}, respectively.
 $(e_1, e_2, e_3)$ is the canonical basis of $\mathbb{R}^3$, $\mathbf{I}_{d}$ ($d = 3,5$)
 is a $d\times d$ unit matrix, $y_i$ denotes the $i$-th component of $y\in \mathbb{ R}^3$, and
\begin{align*}
B_1 =\left(\begin{array}{ccc}
 0 & 0 & 0 \\
0 & 0 &  1\\
0 & -1 &  0
\end{array}
\right), \quad
 B_2  =\left(\begin{array}{ccc}
 0 & 0 & -1 \\
0 & 0 &  0\\
1 & 0 &  0
\end{array}
\right), \quad
B_3 =\left(\begin{array}{ccc}
 0 & 1 & 0 \\
-1 & 0 &  0\\
0 & 0 &  0
\end{array}
\right).
\end{align*}

Using these notations we can rewrite the  problem
\eqref{error1}--\eqref{error5} in the form
\begin{align}\label{error6}
  \left\{\begin{aligned}
&  \D^\epsilon \partial_t \W^\epsilon +\sum^{3}_{i=1}\A^\epsilon_i \W^\epsilon_{x_i}
  +\sum^{3}_{i,j=1}\A^\epsilon_{ij} \W^\epsilon_{x_ix_j}=\s^\epsilon(\W^\epsilon),\\
 & \W^\epsilon|_{t=0}= \W^\epsilon_0.
 \end{aligned} \right.
\end{align}
It is not difficult to see that the system for $\W^\epsilon$  in
\eqref{error6} can be reduced to  a quasilinear symmetric
hyperbolic-parabolic one. In fact, if we introduce
\begin{align*}  \A^\epsilon=\left(\begin{array}{cc}
                  \left(\begin{array}{ccc}
                   \frac{\Theta^\epsilon+\theta^0}{(N^\epsilon+\rho^0)^2} & \mathbf{0} & 0  \\
                   \mathbf{0} &  \mathbf{I}_{3} & \mathbf{0}\\
                   {0} & \mathbf{0} & \frac{1}{\Theta^\epsilon+\theta^0}\\
                                       \end{array}\right)   & \mathbf{0} \\
                    \mathbf{0} &   \mathbf{I}_{6}
                                   \end{array}\right),
\end{align*}
which is positively definite when $\|N^\epsilon\|_{L^\infty_T L^\infty_x}\leq
\hat \rho/2$ and   $\|\Theta^\epsilon\|_{L^\infty_T L^\infty_x}\leq
\hat \theta/2$,
 then $\tilde{\A}_0^\epsilon=\A^\epsilon\D^\epsilon$ and $\tilde{\A}_i^\epsilon = \A^\epsilon\A^\epsilon_i$
 are positive symmetric on $[0,T]$ for all $1\leq i\leq 3.$  Moreover, the assumptions that $\mu>0, 2\mu+3\lambda>0$, and $\kappa>0$ imply that
 $$\mathcal{A}^\epsilon=\sum^{3}_{i,j=1}\A^\epsilon\A^\epsilon_{ij}\W^\epsilon_{x_ix_j}$$ is an elliptic operator.
  Thus, we can apply the result  of Vol'pert and
Hudiaev~\cite{VH}    to obtain   the following local existence for the problem \eqref{error6}.

\begin{prop} \label{Pb}
Let $s>7/2 $ be an integer and $(\rho^0_0, \u^0_0, \theta^0_0,  \H_0^0)$ satisfy the conditions in Proposition \ref{Pa}.
 Assume that the initial data $(N^\epsilon_0, \U^\epsilon_0, \Theta^\epsilon_0, \F_0^\epsilon, \G_0^\epsilon)$ satisfy
 $N^\epsilon_0, \U^\epsilon_0,\Theta^\epsilon_0, \F^\epsilon_0,  \G^\epsilon_0\in
  H^s(\mathbb{T}^3)$,  $\dv \G^\epsilon_0 =0$, and
  \begin{gather*}
 \quad
\|N^\epsilon_0\|_s\leq \delta,\quad\|\Theta^\epsilon_0\|_s\leq \delta
\end{gather*}
  for some  constant $\delta>0$.
Then there exist positive constants $T^\epsilon\,(0<T^\epsilon\leq +\infty)$  and $K$ such that the Cauchy problem
\eqref{error6} has a unique classical solution $(N^\epsilon,
\U^\epsilon, \Theta^\epsilon, \F^\epsilon, \G^\epsilon)$ satisfying $\dv \G^\epsilon =0$ and
\begin{gather*}
 N^\epsilon, \F^\epsilon, \G^\epsilon \in
C^l([0,T^\epsilon),H^{s-l}), \
\U^\epsilon, \Theta^\epsilon \in  C^l([0,T^\epsilon),H^{s-2l}),\  l=0,1; \\
   \|(N^\epsilon,
\U^\epsilon, \Theta^\epsilon, \F^\epsilon, \G^\epsilon)(t)\|_{s}\leq
   K\delta , \ \  t\in [0,T^\epsilon).
  \end{gather*}
 \end{prop}

 Note that  for smooth solutions, the electromagnetic
fluid system \eqref{fya}--\eqref{fye} with the initial data \eqref{fyf} are equivalent to
\eqref{error1}--\eqref{error5} or \eqref{error6} on $[0,T], T<\min\{T^\epsilon,T_*\}$. Therefore, in order to obtain  the convergence of
 electromagnetic fluid equations \eqref{fya}--\eqref{fye} to the full
 compressible magnetohydrodynamic equations \eqref{fza}--\eqref{fzd},
  we only need to establish  uniform decay estimates with respect
  to the parameter $\epsilon$ of the solution to the error system \eqref{error6}.
This will be achieved by the elaborate energy method presented in next section.

\renewcommand{\theequation}{\thesection.\arabic{equation}}
\setcounter{equation}{0}
\section{Uniform energy estimates and proof of Theorem \ref{th}} \label{S3}
 In this section we derive  uniform decay estimates with respect to the parameter
  $\epsilon$ of the solution to the problem \eqref{error6} and
  justify rigorously the convergence of electromagnetic fluid system to the
 full compressible magnetohydrodynamic equations \eqref{fza}--\eqref{fzd}.
Here we adapt and modify some techniques developed in \cite{PW,JL} and put main efforts on the estimates
of higher order nonlinear terms.

We first establish the convergence rate of the error equations by establishing the \emph{a priori} estimates
uniformly in $\epsilon$. For presentation conciseness, we define
\begin{align*}
 &\|\mathcal{E}^\epsilon(t)\|^2_s\    := \|(N^\epsilon,\U^\epsilon,\Theta^\epsilon, \G^\epsilon)(t)\|^2_{s},\\
 &\v \mathcal{E}^\epsilon(t)\v ^2_s :=\|\mathcal{E}^\epsilon(t)\| ^2_s+ {\epsilon }\Vert
\F^\epsilon (t)\Vert^2_{s},\\
& \v\mathcal{E}^\epsilon\v_{s,T}\ : =\sup_{0<t\leq T}\v\mathcal{E}^\epsilon(t)\v_s.
\end{align*}

The crucial estimate of our paper is the following decay   result on the error system
\eqref{error1}--\eqref{error4}.

\begin{prop}\label{P31}
   Let $s>7/2$ be an integer and assume that the initial data  $( N^\epsilon_0,\U^\epsilon_0,\Theta^\epsilon_0, \F^\epsilon_0,\G^\epsilon_0)$ satisfy
\begin{align}\label{ww}
\|(N^\epsilon_0,\U^\epsilon_0,\Theta^\epsilon_0, \G^\epsilon_0) \|^2_{s}+ \epsilon \Vert
\F^\epsilon_0 \Vert^2_{s} =\v \mathcal{E}^\epsilon(t=0)\v^2_{s} \leq M_0{\epsilon}^2
 \end{align}
for sufficiently small $\epsilon$ and   some constant $M_0>0$   independent of $\epsilon$.
Then, for any $T_0\in (0, T_*)$, there exist
  two constants $ M_1 > 0$  and $\epsilon_1 > 0$  depending only on $T_0$, such that
for all $\epsilon\in (0,\epsilon_1]$, it holds that $T^\epsilon\geq T_0$
and the solution $( N^\epsilon, \U^\epsilon,\Theta^\epsilon, \F^\epsilon, \G^\epsilon)$ of the problem
\eqref{error1}--\eqref{error5}, well-defined in $[0, T_0]$, enjoys that
 \begin{align}\label{www}
   \v \mathcal{E}^\epsilon\v _{s,T_0} \leq M_1 {\epsilon}.
 \end{align}
   \end{prop}

Once this proposition is established, the proof of Theorem \ref{th} is a direct procedure. In fact, we have

\begin{proof}[Proof of Theorem \ref{th}]
Suppose that Proposition
\ref{P31} holds. According to the definition of the error functions $(N^\epsilon,\U^\epsilon,\Theta^\epsilon,\F^\epsilon,\G^\epsilon)$
and the regularity of $(\rho^0,\u^0,\theta^0,\H^0)$, the error system \eqref{error1}--\eqref{error4}  and the
primitive system \eqref{fya}--\eqref{fye} are equivalent on $[0,T]$ for some $T>0$.
Therefore  the assumption \eqref{ivda} in Theorem \ref{th} imply the assumption  \eqref{ww} in Proposition
\ref{P31}, and hence  \eqref{www} implies \eqref{iivda}.
\end{proof}

Therefore, our main goal next is to prove Proposition \ref{P31} which can be approached by the following a priori
estimates. For some given  $\hat T<1$ and any $\tilde T<\hat T$ independent of $\epsilon$, we denote
$T \equiv  T_\epsilon = \min\{\tilde T, T^\epsilon \}$.

\begin{lem} Let the assumptions in Proposition \ref{P31} hold. Then, for all $0<t<T $ and sufficiently small $\epsilon$,
there exist two positive constants $\delta_1$ and $\delta_2$, such that
   \begin{align}\label{H2}
     &\v \mathcal{E}^\epsilon(t)\v ^2_s    +   \int^t_0\left\{\delta_1\|  \nabla \U^\epsilon\| ^2_{s}  +\delta_2\| \nabla \Theta^\epsilon\| ^2_{s}
+\frac{1}{4}\| \F^\epsilon\| ^2_{s}\right\}(\tau)\dif\tau\nonumber\\
     \leq  &  \v \mathcal{E}^\epsilon(t=0)\v^2 _{s} + C\int^t_0\left
     \{(\| \mathcal{E}^\epsilon\| ^{2s}_{s}+\| \mathcal{E}^\epsilon\| ^2_{s}+1)\| \mathcal{E}^\epsilon\| ^2_{s}\right\}(\tau)\dif \tau
     +C\epsilon^2.
          \end{align}
\end{lem}
\begin{proof}
   Let $0\leq|\alpha|\leq s$.  In the following arguments the commutators will disappear in the case of $|\alpha|=0$.

    Applying the operator  $\partial^\alpha_x$ to \eqref{error1},
   multiplying  the resulting equation
    by $ \partial^\alpha_x N^\epsilon,$ and integrating over $\mathbb{T}^3$, we obtain that
    \begin{align}\label{hn1}
  \frac12\frac{\rm d}{{\rm d}t}\left\langle  \pa N^\epsilon, \pa N^\epsilon \right\rangle
 = & -\left\langle \pa([(\U^\epsilon+\u^0)\cdot\nabla] N^\epsilon),\pa N^\epsilon\right\rangle\nonumber\\
   &-\left\langle \pa((N^\epsilon+\rho^0)\dv \U^\epsilon), \pa N^\epsilon\right\rangle\nonumber\\
  & +\left\langle \pa((-N^\epsilon \dv \u^0-\nabla \rho^0\cdot \U^\epsilon),\pa N^\epsilon\right\rangle.
    \end{align}

Next we bound every term on the right-hand side of \eqref{hn1}. By the regularity of $\u^0$,
Cauchy-Schwarz's inequality, and Sobolev's imbedding, we have
\begin{align}\label{hn2}
 & \ \ \ \ \langle \partial^\alpha_x([(\U^\epsilon+\u^0)\cdot \nabla]N^\epsilon),\partial^\alpha_x N^\epsilon\rangle\nonumber\\
& = \langle [(\U^\epsilon+\u^0)\cdot \nabla]\partial^\alpha_x N^\epsilon , \partial^\alpha_x N^\epsilon\rangle\
  +\big\langle \mathcal{H}^{(1)},\partial^\alpha_x N^\epsilon \big\rangle\nonumber\\
  & = -\frac12 \langle \dv(\U^\epsilon+\u^0) \partial^\alpha_x N^\epsilon , \partial^\alpha_x N^\epsilon \rangle\
   +\big\langle \mathcal{H}^{(1)},\partial^\alpha_x N^\epsilon \big \rangle\nonumber\\
 & \leq C(\| \mathcal{E}^\epsilon(t)\| _{s}+1)\| \partial^\alpha_x N^\epsilon\| ^2   +\|  \mathcal{H}^{(1)}\| ^2,
\end{align}
where the commutator
\begin{align*}
 \mathcal{H}^{(1)}  =\partial^\alpha_x([(\U^\epsilon+\u^0)\cdot \nabla]N^\epsilon)-[(\U^\epsilon+\u^0)\cdot \nabla]\partial^\alpha_x N^\epsilon
\end{align*}
can be bounded as follows:
\begin{align}\label{ca}
  \big\|\mathcal{H}^{(1)} \big\| &\leq C( \| D_x^1(\U^\epsilon+\u^0)\| _{L^\infty}\| D_x^sN^\epsilon\|
+\| D_x^1N^\epsilon\| _{L^\infty}\| D^{s-1}_x(\U^\epsilon+\u^0)\|) \nonumber\\
   & \leq C(\| \mathcal{E}^\epsilon(t)\| _{s}^2+\| \mathcal{E}^\epsilon(t)\| _{s}).
\end{align}
Here we have used  the Moser-type and Cauchy-Schwarz's inequalities, the regularity of $\u^0$ and Sobolev's imbedding.

Similarly, the second term on the right-hand side of \eqref{hn1} can bounded as follows.
\begin{align}\label{hn3}
  & \left\langle \pa((N^\epsilon+\rho^0)\dv \U^\epsilon), \pa N^\epsilon\right\rangle\nonumber\\
 = \ &\langle (N^\epsilon+\rho^0) \partial^\alpha_x\dv \U^\epsilon , \partial^\alpha_x N^\epsilon\rangle
  +\big\langle \mathcal{H}^{(2)},\partial^\alpha_x N^\epsilon \big\rangle\nonumber\\
  \leq \  & \eta_1 \|\nabla \partial^\alpha_x \U^\epsilon\|^2+ C_{\eta_1} \| \partial^\alpha_x N^\epsilon\| ^2   +\big\|  \mathcal{H}^{(2)}\big\| ^2
\end{align}
for any $\eta_1>0$, where the commutator
\begin{align*}
 \mathcal{H}^{(2)}  =\pa((N^\epsilon+\rho^0)\dv \U^\epsilon) - (N^\epsilon+\rho^0) \partial^\alpha_x\dv \U^\epsilon
\end{align*}
can be estimated by
\begin{align}\label{cb}
\big\|\mathcal{H}^{(2)} \big\| &\leq C( \| D_x^1(N^\epsilon+\rho^0)\| _{L^\infty}\| D_x^s\U^\epsilon\|
+\| D_x^1\U^\epsilon\| _{L^\infty}\| D^{s-1}_x(N^\epsilon+\rho^0)\| )\nonumber\\
   & \leq C(\| \mathcal{E}^\epsilon(t)\| _{s}^2+\| \mathcal{E}^\epsilon(t)\| _{s}).
\end{align}

By the Moser-type and Cauchy-Schwarz's inequalities, and the regularity of $\u^0$ and $\rho^0$, we
can control the third term on the right-hand side of \eqref{hn1} by
\begin{align}\label{hn4}
\left|\left\langle \pa(-N^\epsilon \dv \u^0-\nabla \rho^0\cdot \U^\epsilon),\pa N^\epsilon\right\rangle\right|
\leq C( \|\pa  N^\epsilon\|^2+\|\pa \U^\epsilon\|^2).
\end{align}
Substituting \eqref{hn2}--\eqref{hn4} into \eqref{hn1}, we conclude that
 \begin{align}\label{HN1}
  \frac12\frac{\rm d}{{\rm d}t}\left\langle  \pa N^\epsilon, \pa N^\epsilon \right\rangle
 \leq \ &  \eta_1 \|\nabla \partial^\alpha_x \U^\epsilon\|^2+ C_{\eta_1} \| \partial^\alpha_x N^\epsilon\| ^2 \nonumber\\
 & + C\big[(\| \mathcal{E}^\epsilon(t)\| _{s}+1)\| \partial^\alpha_x N^\epsilon\| ^2+  \| \mathcal{E}^\epsilon(t)\| _{s}^4+ \epsilon^2\big].
  \end{align}

   Applying the operator $\partial^\alpha_x$ to \eqref{error2}, multiplying  the resulting equation
    by $ \partial^\alpha_x\U^\epsilon$, and integrating over $\mathbb{T}^3$, we obtain that
    \begin{align}\label{HU2}
  & \frac12 \frac{\dif}{\dif  t}\langle \partial^\alpha_x\U^\epsilon, \partial^\alpha_x\U^\epsilon \rangle
  + \langle \partial^\alpha_x([(\U^\epsilon+\u^0)\cdot \nabla]\U^\epsilon),\partial^\alpha_x\U^\epsilon\rangle\nonumber\\
   &   +\left\langle \partial^\alpha_x\nabla \Theta^\epsilon, \partial^\alpha_x\U^\epsilon\right\rangle
  +  \left\langle \partial^\alpha_x\left( \frac{\Theta^\epsilon+\theta^0}{N^\epsilon+\rho^0}\nabla N^\epsilon\right), \partial^\alpha_x\U^\epsilon\right\rangle\nonumber\\
 & -  \left\langle   \partial^\alpha_x\left( \frac{1}{N^\epsilon+\rho^0} \dv\Psi(\U^\epsilon)\right) , \partial^\alpha_x \U^\epsilon\right\rangle\nonumber\\
   =     &    -  \left \langle\partial^\alpha_x\left[(\U^\epsilon \cdot \nabla )\u^0\right],   \partial^\alpha_x \U^\epsilon\right\rangle
    -  \left\langle
\partial^\alpha_x\left\{\frac{1}{ \rho^0} \cu
\H^0\times \H^0\right\},\partial^\alpha_x\U^\epsilon \right\rangle\nonumber\\
 &    + \left \langle\partial^\alpha_x\left\{\left[\frac{\Theta^\epsilon+\theta^0}{N^\epsilon+\rho^0}-\frac{\theta^0}{\rho^0}\right] \nabla \rho^0\right\},
   \partial^\alpha_x \U^\epsilon\right\rangle\nonumber\\
   &  +   \left \langle\partial^\alpha_x\left\{\left[\frac{1}{N^\epsilon+\rho^0}-\frac{1}{\rho^0}\right]\dv\Psi(\u^0)\right\},
   \partial^\alpha_x \U^\epsilon\right\rangle\nonumber\\
  & + \left\langle\partial^\alpha_x\left\{ \frac{1}{N^\epsilon+\rho^0}[\F^\epsilon+\u^0\times \G^\epsilon+\U^\epsilon\times \H^0]\times \H^0\right\},
   \partial^\alpha_x\U^\epsilon \right\rangle\nonumber\\
   & + \left\langle \partial^\alpha_x\left\{ \frac{1}{N^\epsilon+\rho^0} [\F^\epsilon+\u^0\times \G^\epsilon+\U^\epsilon\times\H^0]\times \G^\epsilon\right\}
   , \partial^\alpha_x\U^\epsilon \right\rangle\nonumber\\
   & + \left\langle \partial^\alpha_x\left\{\frac{1}{N^\epsilon+\rho^0}
   (\U^\epsilon\times \G^\epsilon)\times (\G^\epsilon+\H^0)\right\} , \partial^\alpha_x\U^\epsilon \right\rangle \nonumber\\
 := &  \sum^{7}_{i=1}\mathcal{R}^{(i)}.
  \end{align}

We first bound the terms on the left-hand side of \eqref{HU2}. Similar to \eqref{hn2} we infer that
\begin{align}\label{hu1}
  & \ \ \ \ \langle \partial^\alpha_x([(\U^\epsilon+\u^0)\cdot \nabla]\U^\epsilon),\partial^\alpha_x\U^\epsilon\rangle \nonumber\\
  & = \langle [(\U^\epsilon+\u^0)\cdot \nabla]\partial^\alpha_x \U^\epsilon , \partial^\alpha_x \U^\epsilon\rangle\
  +\big\langle \mathcal{H}^{(3)},\partial^\alpha_x \U^\epsilon \big\rangle\nonumber\\
  & = -\frac12 \langle \dv(\U^\epsilon+\u^0) \partial^\alpha_x \U^\epsilon , \partial^\alpha_x \U^\epsilon \rangle\
   +\big\langle \mathcal{H}^{(3)},\partial^\alpha_x \U^\epsilon \big \rangle\nonumber\\
 & \leq C(\| \mathcal{E}^\epsilon(t)\| _{s}+1)\| \partial^\alpha_x \U^\epsilon\| ^2   +\big\|  \mathcal{H}^{(3)}\big\| ^2,
\end{align}
where the commutator
\begin{align*}
 \mathcal{H}^{(3)}  =\partial^\alpha_x([(\U^\epsilon+\u^0)\cdot \nabla]\U^\epsilon)-[(\U^\epsilon+\u^0)\cdot \nabla]\partial^\alpha_x \U^\epsilon
\end{align*}
can be bounded by
\begin{align}\label{hu2}
\big\|\mathcal{H}^{(3)} \big\| &\leq C( \| D_x^1(\U^\epsilon+\u^0)\| _{L^\infty}\| D_x^s\U^\epsilon\|
+\| D_x^1\U^\epsilon\| _{L^\infty}\| D^{s-1}_x(\U^\epsilon+\u^0)\|) \nonumber\\
   & \leq C(\| \mathcal{E}^\epsilon(t)\| _{s}^2+\| \mathcal{E}^\epsilon(t)\| _{s}).
\end{align}

By Holder's inequality, we have
\begin{align}\label{hu3}
 \left\langle \partial^\alpha_x\nabla \Theta^\epsilon, \partial^\alpha_x\U^\epsilon\right\rangle
\leq \eta_2    \|\partial^\alpha_x\nabla \Theta^\epsilon\|^2+C_{\eta_2}\|\partial^\alpha_x\U^\epsilon\|^2
\end{align}
for any $\eta_2>0$.
For the fourth term on the left-hand side of \eqref{HU2}, similar to \eqref{hn3},   we integrate by parts to deduce that
\begin{align}\label{hu4}
 & \quad\,  \left\langle \partial^\alpha_x\left( \frac{\Theta^\epsilon+\theta^0}{N^\epsilon+\rho^0}\nabla N^\epsilon\right), \partial^\alpha_x\U^\epsilon\right\rangle \nonumber\\
& = \left\langle  \frac{\Theta^\epsilon+\theta^0}{N^\epsilon+\rho^0}\partial^\alpha_x \nabla N^\epsilon , \partial^\alpha_x\U^\epsilon\right\rangle
 + \big\langle \mathcal{H}^{(4)}, \partial^\alpha_x\U^\epsilon\big\rangle\nonumber\\
&   =  -\left\langle\partial^\alpha_x   N^\epsilon , \dv\left(\frac{\Theta^\epsilon+\theta^0}{N^\epsilon+\rho^0}\partial^\alpha_x\U^\epsilon\right)  \right\rangle\
  +\big\langle \mathcal{H}^{(4)},\partial^\alpha_x \U^\epsilon \big\rangle\nonumber\\
&   \leq   \eta_3 \|\nabla \partial^\alpha_x \U^\epsilon\|^2+ C_{\eta_3} \| \partial^\alpha_x N^\epsilon\| ^2 +C \| \mathcal{E}^\epsilon(t)\| _{s}^4
 +\big\|  \mathcal{H}^{(4)}\big\| ^2
\end{align}
for any $\eta_3>0$,
where the commutator
\begin{align*}
 \mathcal{H}^{(4)}  =
  \partial^\alpha_x\left( \frac{\Theta^\epsilon+\theta^0}{N^\epsilon+\rho^0}\nabla N^\epsilon\right)
  -   \frac{\Theta^\epsilon+\theta^0}{N^\epsilon+\rho^0}\partial^\alpha_x \nabla N^\epsilon
\end{align*}
can be bounded  as follows by using \eqref{mb} and \eqref{mo}, and Cauchy-Schwarz's inequality:
 \begin{align}\label{hu44}
\big\|\mathcal{H}^{(4)}\big\| &\leq C\left( \left\| D_x^1\left( \frac{\Theta^\epsilon+\theta^0}{N^\epsilon+\rho^0}\right)\right\|_{L^\infty}\| D_x^sN^\epsilon\|
+\| D_x^1N^\epsilon\| _{L^\infty}\left\| D^{s-1}_x\left( \frac{\Theta^\epsilon+\theta^0}{N^\epsilon+\rho^0}\right)\right\|\right) \nonumber\\
   & \leq C(\| \mathcal{E}^\epsilon(t)\|_{s}^{2(s+1)}+\| \mathcal{E}^\epsilon(t)\| _{s}^2+\| \mathcal{E}^\epsilon(t)\| _{s}).
\end{align}

For the fifth term on the left-hand side of \eqref{HU2}, we integrate by parts to deduce
\begin{align} \label{hu5}
   &-  \left\langle   \partial^\alpha_x\left( \frac{1}{N^\epsilon+\rho^0} \dv\Psi(\U^\epsilon)\right) , \partial^\alpha_x \U^\epsilon\right\rangle\nonumber\\
  =& -  \left\langle  \frac{1}{N^\epsilon+\rho^0}  \partial^\alpha_x \dv\Psi(\U^\epsilon)  , \partial^\alpha_x\U^\epsilon\right\rangle
   -  \big\langle\mathcal{H}^{(5)}, \partial^\alpha_x\U^\epsilon\big\rangle,
  \end{align}
where the commutator
\begin{align*}
   \mathcal{H}^{(5)}=  \partial^\alpha_x\left( \frac{1}{N^\epsilon+\rho^0} \dv\Psi(\U^\epsilon)\right)
  - \frac{1}{N^\epsilon+\rho^0}  \partial^\alpha_x \dv\Psi(\U^\epsilon).
\end{align*}
By the Moser-type and Cauchy-Schwarz inequalities, the regularity of $\rho^0$ and the
positivity of $N^\epsilon+\rho_0$,  the definition of  $\Psi(\U^\epsilon)$ and Sobolev's imbedding, we find that
 \begin{align}\label{hu6}
 & \ \ \   \big|\big\langle\mathcal{H}^{(5)},
 \partial^\alpha_x\U^\epsilon\big\rangle\big|\nonumber\\
   & \leq  \big\|\mathcal{H}^{(5)}\big\|\cdot \| \partial^\alpha_x\U^\epsilon \|  \nonumber\\
   & \leq C \left(\left\Vert D_x^1\left(\frac{1}{N^\epsilon+\rho^0}\right)\right\Vert_{L^\infty}\|  \dv\Psi(\U^\epsilon)\| _{s-1}
   +\| \dv\Psi(\U^\epsilon)\| _{L^\infty}\left\Vert\frac{1}{N^\epsilon+\rho^0}\right\Vert_{s}\right)\| \partial^\alpha_x\U^\epsilon \| \nonumber\\
   &\leq \eta_4 \| \nabla \U^\epsilon\| ^2_{s}+C_{\eta_4} (\| \mathcal{E}^\epsilon(t)\| ^2_{s}+1)(\| \partial^\alpha_x
   \U^\epsilon\| ^2+\| \partial^\alpha_x  N^\epsilon\| ^2+\| \mathcal{E}^\epsilon(t)\|_{s}^{s})
 \end{align}
 for any $\eta_4>0$, where we have used the assumption $s>3/2+2$
and the imbedding $H^l(\mathbb{T}^3)\hookrightarrow L^\infty (\mathbb{R}^3)$ for $l>3/2$.
By virtue of the definition of $\Psi(\U^\epsilon)$ and partial integrations, the first term on the right-hand side of \eqref{hu5}
can be rewritten as
    \begin{align}\label{hu7}
 &  -  \left\langle  \frac{1}{N^\epsilon+\rho^0}  \partial^\alpha_x \dv\Psi(\U^\epsilon)  , \partial^\alpha_x\U^\epsilon\right\rangle\nonumber\\
 =\ &2\mu \left\langle \frac{1}{N^\epsilon+\rho^0}\partial^\alpha_x\mathbb{D}(\U^\epsilon),\partial^\alpha_x\mathbb{D}(\U^\epsilon)\right\rangle
+\lambda \left\langle \frac{1}{N^\epsilon+\rho^0}\partial^\alpha_x\dv \U^\epsilon ,\partial^\alpha_x \dv \U^\epsilon \right\rangle
\nonumber\\
& +2\mu \left\langle \nabla\left(\frac{1}{N^\epsilon+\rho^0}\right)\otimes \partial^\alpha_x\U^\epsilon,\partial^\alpha_x\mathbb{D}(\U^\epsilon)\right\rangle\nonumber\\
&+\lambda \left\langle \nabla\left(\frac{1}{N^\epsilon+\rho^0}\right)\cdot\partial^\alpha_x \U^\epsilon ,\partial^\alpha_x \dv \U^\epsilon \right\rangle
\nonumber\\
:=\ & \sum^4_{i=1}\mathcal{I}^{(i)}.
\end{align}

Recalling the facts that $\mu>0$ and $2\mu+3\lambda>0$, and the
positivity of $N^\epsilon+\rho_0$,  the first two terms $\mathcal{I}^{(1)}$ and $\mathcal{I}^{(2)}$
can be bounded as follows:
\begin{align}\label{hu8}
 \mathcal{I}^{(1)}+\mathcal{I}^{(2)} &=\int  \frac{1}{N^\epsilon+\rho^0}\big\{2\mu|\pa\mathbb{D}
 (\U^\epsilon)|^2+\lambda|\pa\mbox{tr}\mathbb{D}(\U^\epsilon)|^2\big\}{\rm d}x\nonumber\\
&  \geq 2 \mu \int \frac{1}{N^\epsilon+\rho^0}\left(|\pa\mathbb{D}(\U^\epsilon)|^2-\frac13 |\pa\mbox{tr}\mathbb{D}(\U^\epsilon)|^2\right){\rm d}x\nonumber\\
  & =   \mu \int \frac{1}{N^\epsilon+\rho^0}\left(|\pa\nabla\U^\epsilon|^2+\frac13 |\pa\dv  \U^\epsilon|^2\right){\rm d}x\nonumber\\
  & \geq   \mu \int \frac{1}{N^\epsilon+\rho^0}|\pa\nabla\U^\epsilon|^2{\rm d}x.
\end{align}

  By virtue of Cauchy-Schwarz's inequality, the regularity of $\rho^0$ and the
positivity of $N^\epsilon+\rho_0$, the terms $\mathcal{I}^{(3)}$ and $\mathcal{I}^{(4)}$ can be bounded by
 \begin{align}\label{hu9}
|\mathcal{I}^{(3)}|+|\mathcal{I}^{(4)}|
    &\leq \eta_5 \|\nabla \pa\U^\epsilon\|^2 +C_{\eta_5} (\|\mathcal{E}^\epsilon(t)\|^2_{s}+1)(\|\partial^\alpha_x
   \U^\epsilon\|^2 +\|\partial^\alpha_x N^\epsilon\|^2)
 \end{align}
 for any $\eta_5>0$, where the assumption $s>3/2+2$ has been used.

Substituting \eqref{hu1}--\eqref{hu9} into \eqref{HU2}, we conclude that
\begin{align}\label{HU20}
&   \frac12\frac{\dif}{\dif t}\langle \partial^\alpha_x\U^\epsilon,
\partial^\alpha_x\U^\epsilon \rangle
  +  \int \frac{\mu}{N^\epsilon+\rho^0} |\nabla \partial^\alpha_x \U^\epsilon|^2 \dif x-
  (\eta_1+\eta_3+\eta_4+\eta_5)\|\nabla \partial^\alpha_x \U^\epsilon\|^2\nonumber\\
  \leq\ &
   C_{{\eta}}\big\{ (\| \mathcal{E}^\epsilon(t)\| ^2_{s}+1)(\| \partial^\alpha_x
   \U^\epsilon\| ^2+\| \partial^\alpha_x  N^\epsilon\| ^2+\| \mathcal{E}^\epsilon(t)\|_{s}^{s})\big\}\nonumber\\
   & + \eta_2    \|\partial^\alpha_x\nabla \Theta^\epsilon\|^2 +\sum^{7}_{i=1}\mathcal{R}^{(i)}
\end{align}
for some constant $C_{{\eta}}>0$ depending on $\eta_i$ ($i=1,\dots,5$).

 We have to estimate the terms on the right-hand side of \eqref{HU20}.
  In view of the regularity of $(\rho^0,\u^0,\H^0)$,
the positivity of $ N^\epsilon+\rho^0$ and Cauchy-Schwarz's inequality, the first
two terms $\mathcal{R}^{(1)}$ and $\mathcal{R}^{(2)}$ can be controlled by
\begin{align}\label{hu10}
   \big|\mathcal{R}^{(1)}\big|+\big|\mathcal{R}^{(2)}\big| \leq  C(\|\mathcal{E}^\epsilon(t)\|_{s}^2 +1)
   ( \|\partial^\alpha_x N^\epsilon\|^2 +\|\partial^\alpha_x\U^\epsilon\|^2).
\end{align}

For the terms $\mathcal{R}^{(3)}$ and $\mathcal{R}^{(4)}$,  by the regularity of $ \rho^0$ and $\u^0$,
the positivity of $ N^\epsilon+\rho^0$,  Cauchy-Schwarz's inequality and \eqref{ho}, we see that
\begin{align}\label{hu11}
   \big|\mathcal{R}^{(3)}\big|+\big|\mathcal{R}^{(4)}\big| \leq  C(\|\mathcal{E}^\epsilon(t)\|_{s}^2  +
   C\|\partial^\alpha_x\U^\epsilon\|^2).
\end{align}

For the fifth term $\mathcal{R}^{(5)}$, we utilize the positivity of $N^\epsilon+\rho^0$ to deduce that
\begin{align}\label{hu12}
   \mathcal{R}^{(5)}& =  \left\langle\partial^\alpha_x \F^\epsilon\times
\frac{\H^0}{N^\epsilon+\rho^0},
   \partial^\alpha_x\U^\epsilon \right\rangle
+  \big\langle \mathcal{H}^{(6)},
\partial^\alpha_x\U^\epsilon \big\rangle + \sigma\mathcal{R}^{(5_1)}\nonumber\\
 & \leq \frac{1}{16} \| \partial^\alpha_x \F^\epsilon\| ^2+
 C\| \partial^\alpha_x\U^\epsilon\| ^2+  \big\langle \mathcal{H}^{(6)}, \partial^\alpha_x\U^\epsilon\big\rangle +  \mathcal{R}^{(5_1)},
\end{align}
where
\begin{align*}
 \mathcal{H}^{(6)}  =\partial^\alpha_x\left\{
\frac{\F^\epsilon}{N^\epsilon+\rho^0}\times
\H^0\right\}-\partial^\alpha_x \F^\epsilon\times
\frac{\H^0}{N^\epsilon+\rho^0}\nonumber\
\end{align*}
and
\begin{align*}
\mathcal{R}^{(5_1)}=\left\langle\partial^\alpha_x\left\{
\frac{\sigma}{N^\epsilon+\rho^0}[\u^0\times
\G^\epsilon+\U^\epsilon\times \H^0]\times \H^0\right\},
   \partial^\alpha_x\U^\epsilon \right\rangle.
\end{align*}

If we make use of the Moser-type inequality, \eqref{mo} and the regularity of $\rho^0$ and $\H^0$, we obtain that
 \begin{align}\label{hu13}
 &  \ \ \   \big|\big\langle\mathcal{H}^{(6)},
  \partial^\alpha_x\U^\epsilon\big\rangle\big| \leq  \big\|\mathcal{H}^{(6)}\big\|\cdot \| \partial^\alpha_x\U^\epsilon \|  \nonumber\\
   & \leq C \left[\left\Vert D_x^1\left(\frac{\H^0}{N^\epsilon+\rho^0}\right)\right\Vert_{L^\infty}\| \F^\epsilon\| _{s-1}
   +\| \F^\epsilon\| _{L^\infty}\left\Vert\frac{\H^0}{N^\epsilon+\rho^0}\right\Vert_{s}\right]\| \partial^\alpha_x\U^\epsilon \| \nonumber\\
   &\leq \eta_6 \| \F^\epsilon\| ^2_{s-1}+C_{\eta_6} (\| \mathcal{E}^\epsilon(t)\| ^{2(s+1)}_{s}+1) \| \partial^\alpha_x
   \U^\epsilon\| ^2
 \end{align}
 for any $\eta_6>0$. Recalling the regularity of $\u^0$ and $\H^0$, \eqref{ma} and \eqref{mo} and H\"{o}lder's
inequality, we find that
\begin{align}\label{hu14}
\big|\mathcal{R}^{(5_1)}\big|\leq  C(\| \mathcal{E}^\epsilon(t)\| _{s}^s
+1)(\| \partial^\alpha_x
N^\epsilon\| ^2+\| \partial^\alpha_x\U^\epsilon\| ^2+\| \partial^\alpha_x\G^\epsilon\| ^2).
\end{align}

For the sixth term $\mathcal{R}^{(6)}$ we again make use of the positivity of
$N^\epsilon+\rho^0$ and Sobolev's imbedding to infer that
\begin{align}\label{hu15}
   \mathcal{R}^{(6)}& = \left\langle\partial^\alpha_x \F^\epsilon\times
\frac{\G^\epsilon}{N^\epsilon+\rho^0},
   \partial^\alpha_x\U^\epsilon \right\rangle
+  \big\langle \mathcal{H}^{(7)},\partial^\alpha_x\U^\epsilon\big\rangle +  \mathcal{R}^{(6_1)}\nonumber\\
 & \leq \frac{1}{16} \| \partial^\alpha_x \F^\epsilon\| ^2+
 C\| \mathcal{E}^\epsilon(t)\| _{s}^2
 \| \partial^\alpha_x\U^\epsilon\| ^2+ \big\langle \mathcal{H}^{(7)},  \partial^\alpha_x\U^\epsilon \big\rangle +  \mathcal{R}^{(6_1)},
\end{align}
where
\begin{align*}
 \mathcal{H}^{(7)}  =\partial^\alpha_x\left\{
\frac{\F^\epsilon}{N^\epsilon+\rho^0}\times
\G^\epsilon\right\}-\partial^\alpha_x \F^\epsilon\times
\frac{\G^\epsilon}{N^\epsilon+\rho^0}\nonumber\
\end{align*}
and
\begin{align*}
\mathcal{R}^{(6_1)}=\left\langle\partial^\alpha_x\left\{
\frac{\sigma}{N^\epsilon+\rho^0}[\u^0\times
\G^\epsilon+\U^\epsilon\times \H^0]\times \G^\epsilon\right\},
   \partial^\alpha_x\U^\epsilon \right\rangle.
\end{align*}
From the H\"{o}lder's and Moser-type inequalities we get
 \begin{align}\label{hu16}
 &  \ \ \ \  \big|\big\langle\mathcal{H}^{(7)},
  \partial^\alpha_x\U^\epsilon\big\rangle\big| \nonumber\\
  &\leq \big \| \mathcal{H}^{(7)}\big\| \cdot \| \partial^\alpha_x\U^\epsilon \|  \nonumber\\
  & \leq C \left[\left\Vert D_x^1\left(\frac{\G^\epsilon}{N^\epsilon+\rho^0}\right)\right\Vert_{L^\infty}\| \F^\epsilon\| _{s-1}
   +\| \F^\epsilon\| _{L^\infty}\left\Vert\frac{\G^\epsilon}{N^\epsilon+\rho^0}\right\Vert_{s}\right]\| \partial^\alpha_x\U^\epsilon \| \nonumber\\
   &\leq \eta_7 \| \F^\epsilon\| ^2_{s-1}+C_{\eta_7} (\| \mathcal{E}^\epsilon(t)\| ^{2(s+1)}_{s}+1) \| \partial^\alpha_x
   \U^\epsilon\| ^2
 \end{align}
 for any $\eta_7>0$, while for the term $\mathcal{R}^{(6_1)}$ one has the following estimate
\begin{align}\label{hu17}
\big| \mathcal{R}^{6_1)}\big|\leq C(\| \mathcal{E}^\epsilon(t)\|_{s}^2
+\| \mathcal{E}^\epsilon(t)\| _{s} +1)\| \mathcal{E}^\epsilon(t)\|_{s}^2.
\end{align}

For the last term $\mathcal{R}^{(7)}$, recalling the formula
     $  (\mathbf{a} \times \mathbf{b}) \times \mathbf{c}=
     (\mathbf{a} \cdot \mathbf{c})\mathbf{b}-(\mathbf{b} \cdot \mathbf{c})\mathbf{a}$ and applying
     \eqref{ma}, \eqref{mo}, and H\"{o}lder's inequality,
we easily deduce that
\begin{align}
\big|\mathcal{R}^{(7)}\big| = \ &
\left|\left\langle\partial^\alpha_x\left\{\frac{1}{N^\epsilon+\rho^0}
  \{[\U^\epsilon\cdot(\G^\epsilon+\H^0)]\G^\epsilon-[\G^\epsilon\cdot(\G^\epsilon+\H^0)]\U^\epsilon\}\right\},\partial^\alpha_x\U^\epsilon
   \right\rangle\right| \nonumber\\
 \leq \ &
    C(\| \mathcal{E}^\epsilon(t)\| _{s}^s+1)\| \mathcal{E}^\epsilon(t)\| _{s}^4+\| \mathcal{E}^\epsilon(t)\| _{s}^3.  \label{hu18}
   \end{align}

Substituting  \eqref{hu10}--\eqref{hu18} into \eqref{HU20}, we conclude that
 \begin{align}\label{HU202}
&   \frac12\frac{\dif}{\dif t}\langle \partial^\alpha_x\U^\epsilon,
\partial^\alpha_x\U^\epsilon \rangle
  +  \int \frac{\mu}{N^\epsilon+\rho^0} |\nabla \partial^\alpha_x \U^\epsilon|^2 \dif x-
  (\eta_1+\eta_3+\eta_4+\eta_5)\|\nabla \partial^\alpha_x \U^\epsilon\|^2\nonumber\\
\leq \  &
    \tilde{C}_{{\eta}}\big[(\| \mathcal{E}^\epsilon(t)\| _{s}^{2s}+1)\| \mathcal{E}^\epsilon(t)\| _{s}^4+\| \mathcal{E}^\epsilon(t)\| _{s}^3 + \| \mathcal{E}^\epsilon(t)\| ^2_{s} \big]\nonumber\\
  & +\eta_2    \|\partial^\alpha_x\nabla \Theta^\epsilon\|^2+\left(\eta_6+\eta_7+\frac{1}{8}\right)\| \F^\epsilon\| ^2_{s}.
\end{align}
for some constant $\tilde{C}_{{\eta}}>0$ depending on $\eta_i$ ($i=1,\dots,7$).

  Applying the operator $\partial^\alpha_x$ to \eqref{error22}, multiplying  the resulting equation
    by $ \partial^\alpha_x\Theta^\epsilon$, and integrating over $\mathbb{T}^3$, we arrive at
\begin{align}  \label{HT1}
 & \frac12\frac{\dif}{\dif t}\langle\pa\Theta^\epsilon,\pa\Theta^\epsilon\rangle  +\langle \pa\{[(\U^\epsilon+\u^0)\cdot \nabla]\Theta^\epsilon\},\pa\Theta^\epsilon\rangle    \nonumber\\
& + \left\langle\pa\{(\Theta^\epsilon+\theta^0)\, \dv \U^\epsilon\},\pa\Theta^\epsilon\right\rangle-\left\langle\pa\left\{\frac{\kappa}{N^\epsilon+\rho^0}\Delta \Theta^\epsilon\right\},\pa\Theta^\epsilon\right\rangle\nonumber\\
 =  &-\langle \pa\{(\U^\epsilon \cdot \nabla )\theta^0
  -\Theta^\epsilon \dv \u^0\},\pa\Theta^\epsilon\rangle\nonumber\\
  & +\left\langle\pa\left\{\left[\frac{\kappa}{N^\epsilon+\rho^0}-\frac{\kappa}{\rho^0}\right]\Delta \theta^0 \right\},\pa\Theta^\epsilon\right\rangle   \nonumber\\
  & +\left\langle\pa\left\{\left[\frac{2\mu}{N^\epsilon+\rho^0}-\frac{2\mu}{\rho^0}\right]|\mathbb{D}(\u^0)|^2 \right\},\pa\Theta^\epsilon\right\rangle   \nonumber\\
  & +\left\langle\pa\left\{\left[\frac{\lambda}{N^\epsilon+\rho^0}-\frac{\lambda}{\rho^0}\right](\mbox{tr}\mathbb{D}(\u^0))^2 \right\},\pa\Theta^\epsilon\right\rangle   \nonumber\\
  & +\left\langle\pa\left\{\left[\frac{1}{N^\epsilon+\rho^0}-\frac{1}{\rho^0}\right]|\cu \H^0|^2 \right\},\pa\Theta^\epsilon\right\rangle   \nonumber\\
  &  + \left\langle\pa\left\{\frac{ 2\mu}{N^\epsilon+\rho^0} |\mathbb{D}(\U^\epsilon)|^2+\frac{\lambda}{N^\epsilon+\rho^0}|\mbox{tr}\mathbb{D}(\U^\epsilon)|^2\right\},\pa\Theta^\epsilon\right\rangle\nonumber\\
  &   + \left\langle\pa\left\{\frac{ 4\mu}{N^\epsilon+\rho^0}\mathbb{D}(\U^\epsilon): \mathbb{D}(\u^0)\right\},\pa\Theta^\epsilon\right\rangle\nonumber\\
   &  + \left\langle\pa\left\{\frac{ 2\lambda}{N^\epsilon+\rho^0}\,[\mbox{tr}\mathbb{D}(\U^\epsilon) \mbox{tr}\mathbb{D}(\u^0)]\right\},\pa\Theta^\epsilon\right\rangle\nonumber\\
    &  +\left\langle\pa\left\{ \frac{1}{N^\epsilon+\rho^0}|\F^\epsilon+\U^\epsilon\times \G^\epsilon|^2\right\},\pa\Theta^\epsilon\right\rangle\nonumber\\
    & +\left\langle\pa\left\{ \frac{1}{N^\epsilon+\rho^0} |\u^0\times \G^\epsilon+\U^\epsilon\times \H^0|^2\right\},\pa\Theta^\epsilon\right\rangle\nonumber\\
    &  + \left\langle\pa\left\{\frac{2\F^\epsilon}{N^\epsilon+\rho^0}\cdot
                   [\cu \H^0+\u^0\times \G^\epsilon+\U^\epsilon\times \H^0]\right\},\pa\Theta^\epsilon\right\rangle\nonumber\\
    &  + \left\langle\pa\left\{\frac{2(\U^\epsilon\times \G^\epsilon)}{N^\epsilon+\rho^0}\cdot
                   [\cu \H^0+\u^0\times \G^\epsilon+\U^\epsilon\times \H^0]\right\},\pa\Theta^\epsilon\right\rangle\nonumber\\
      &  +   \left\langle\pa\left\{\frac{2}{N^\epsilon+\rho^0}\cu \H^0\cdot (\u^0\times \G^\epsilon+\U^\epsilon\times \H^0)\right\},\pa\Theta^\epsilon\right\rangle\nonumber\\
      := & \sum^{13}_{i=1}\mathcal{S}^{(i)}.
\end{align}

We first bound the terms on the left-hand side of \eqref{HT1}. Similar to \eqref{hn2}, we  have
\begin{align}\label{ht1}
 &  \ \ \ \ \langle \partial^\alpha_x([(\U^\epsilon+\u^0)\cdot \nabla]\Theta^\epsilon),\partial^\alpha_x\Theta^\epsilon\rangle\nonumber\\
  & = \langle [(\U^\epsilon+\u^0)\cdot \nabla]\partial^\alpha_x \Theta^\epsilon , \partial^\alpha_x \Theta^\epsilon\rangle\
  +\big\langle \mathcal{H}^{(8)},\partial^\alpha_x \Theta^\epsilon \big\rangle\nonumber\\
  & = -\frac12 \langle \dv(\U^\epsilon+\u^0) \partial^\alpha_x \Theta^\epsilon , \partial^\alpha_x \Theta^\epsilon \rangle\
   +\big\langle \mathcal{H}^{(8)},\partial^\alpha_x \Theta^\epsilon \big \rangle\nonumber\\
 & \leq C(\| \mathcal{E}^\epsilon(t)\| _{s}+1)\| \partial^\alpha_x \Theta^\epsilon\| ^2   +\big\|  \mathcal{H}^{(8)}\big\| ^2,
\end{align}
where the commutator
\begin{align*}
 \mathcal{H}^{(8)}  =\partial^\alpha_x([(\U^\epsilon+\u^0)\cdot \nabla]\Theta^\epsilon)-[(\U^\epsilon+\u^0)\cdot \nabla]\partial^\alpha_x \Theta^\epsilon
\end{align*}
can be bounded by
\begin{align}\label{ht2}
  \big\|\mathcal{H}^{(8)} \big\| &\leq C( \| D_x^1(\U^\epsilon+\u^0)\| _{L^\infty}\| D_x^s\U^\epsilon\|
+\| D_x^1\U^\epsilon\| _{L^\infty}\| D^{s-1}_x(\U^\epsilon+\u^0)\|) \nonumber\\
   & \leq C(\| \mathcal{E}^\epsilon(t)\| _{s}^2+\| \mathcal{E}^\epsilon(t)\| _{s}).
\end{align}
The second term on the left-hand side of \eqref{HT1} can bounded, similar to \eqref{hn3}, as follows:
\begin{align}\label{ht3}
  & \left\langle \pa((\Theta^\epsilon+\theta^0)\dv \U^\epsilon), \pa \Theta^\epsilon\right\rangle\nonumber\\
 =\ &\langle (\Theta^\epsilon+\rho^0) \partial^\alpha_x\dv \U^\epsilon , \partial^\alpha_x \Theta^\epsilon\rangle
  +\big\langle \mathcal{H}^{(9)},\partial^\alpha_x \Theta^\epsilon \big\rangle\nonumber\\
  \leq\  & \eta_8 \|\nabla \partial^\alpha_x \U^\epsilon\|^2+ C_{\eta_8} \| \partial^\alpha_x N^\epsilon\| ^2   +\big\|  \mathcal{H}^{(9)}\big\| ^2
\end{align}
for any $\eta_8>0$, where the commutator
\begin{align*}
 \mathcal{H}^{(9)}  =\pa((\Theta^\epsilon+\rho^0)\dv \U^\epsilon) - (\Theta^\epsilon+\theta^0) \partial^\alpha_x\dv \U^\epsilon
\end{align*}
can be controlled as
\begin{align}\label{ht4}
  \big\|\mathcal{H}^{(9)}\big\| &\leq C( \| D_x^1(\Theta^\epsilon+\theta^0)\| _{L^\infty}\| D_x^s\U^\epsilon\|
+\| D_x^1\U^\epsilon\| _{L^\infty}\| D^{s-1}_x(\Theta^\epsilon+\theta^0)\| )\nonumber\\
   & \leq C(\| \mathcal{E}^\epsilon(t)\| _{s}^2+\| \mathcal{E}^\epsilon(t)\| _{s}).
\end{align}

For the fourth term on the left-hand side of \eqref{HT1}, we integrate by parts to deduce that
\begin{align} \label{ht5}
   &- \kappa \left\langle \partial^\alpha_x\left( \frac{1}{N^\epsilon+\rho^0}\Delta \U^\epsilon\right), \partial^\alpha_x\Theta^\epsilon\right\rangle
   \nonumber \\
  & = - \kappa  \left\langle  \frac{1}{N^\epsilon+\rho^0}\Delta \partial^\alpha_x \Theta^\epsilon , \partial^\alpha_x\Theta^\epsilon\right\rangle
   -  \kappa\big\langle\mathcal{H}^{(10)}, \partial^\alpha_x\Theta^\epsilon\big\rangle  \nonumber\\
 & = \kappa \left\langle  \frac{1}{N^\epsilon+\rho^0}\nabla\partial^\alpha_x \Theta^\epsilon , \nabla\partial^\alpha_x\Theta^\epsilon\right\rangle
     \nonumber\\
   & \quad + \kappa\left \langle \nabla \left(\frac{1}{N^\epsilon+\rho^0}\right)   \nabla\partial^\alpha_x \Theta^\epsilon,
\partial^\alpha_x \Theta^\epsilon\right\rangle -  \kappa\big\langle\mathcal{H}^{(10)}, \partial^\alpha_x\Theta^\epsilon\big\rangle,
   \end{align}
where
\begin{align*}
   \mathcal{H}^{(10)}=  \partial^\alpha_x\left( \frac{1}{N^\epsilon+\rho^0}\Delta \Theta^\epsilon\right)
   - \frac{1}{N^\epsilon+\rho^0}\Delta \partial^\alpha_x \Theta^\epsilon.
\end{align*}
By the Moser-type and H\"{o}lder's inequalities, the regularity of $\rho^0$, the
positivity of $N^\epsilon+\rho_0$ and \eqref{mo}, we find that
 \begin{align}\label{ht6}
 & \ \ \   \big|\big\langle\mathcal{H}^{(10)},
 \partial^\alpha_x\Theta^\epsilon\big\rangle\big|
    \leq  \big\|\mathcal{H}^{(10)}\big\|\cdot \| \partial^\alpha_x\Theta^\epsilon \|  \nonumber\\
   & \leq C \left(\left\Vert D_x^1\left(\frac{1}{N^\epsilon+\rho^0}\right)\right\Vert_{L^\infty}\| \Delta \Theta^\epsilon\| _{s-1}
   +\| \Delta \Theta^\epsilon\| _{L^\infty}\left\Vert\frac{1}{N^\epsilon+\rho^0}\right\Vert_{s}\right)\| \partial^\alpha_x\Theta^\epsilon \| \nonumber\\
   &\leq \eta_9 \| \nabla \Theta^\epsilon\| ^2_{s}+C_{\eta_9} (\| \mathcal{E}^\epsilon(t)\| ^{s}_{s}+1)(\| \partial^\alpha_x
   \Theta^\epsilon\| ^2+\| \partial^\alpha_x
  N^\epsilon\| ^2)
 \end{align}
 and
 \begin{align}\label{ht7}
 & \ \ \  \left|\left \langle \nabla \left(\frac{1}{N^\epsilon+\rho^0}\right) \nabla{\partial^\alpha_x \Theta^\epsilon} ,
\partial^\alpha_x \Theta^\epsilon\right\rangle\right|\nonumber\\
   &  \leq\eta_{10}\|\nabla \partial^\alpha_x\Theta^\epsilon\|^2
   + C_{\eta_{10}} \left\Vert \nabla\left(\frac{1}{N^\epsilon+\rho^0}\right)\right\Vert_{L^\infty}^2
   \|\partial^\alpha_x   \Theta^\epsilon\|^2\nonumber\\
   &  \leq \eta_{10}\|\nabla \partial^\alpha_x\Theta^\epsilon\|^2
   + C_{\eta_{10}}  (\|\mathcal{E}^\epsilon(t)\|^2_{s}+1)\| \partial^\alpha_x\Theta^\epsilon\|^2
   \end{align}
 for any $\eta_9>0 $ and $\eta_{10}>0$, where
 we have used the assumption  $s>3/2+2$ in the derivation of \eqref{ht6}
and the imbedding $H^l ( \mathbb{T}^3)\hookrightarrow L^\infty (\mathbb{R}^3)$ for $l>3/2$.

 Now, we estimate every term on the right-hand side of \eqref{HT1}.
 By virtue of the regularity of $ \theta^0$ and $\u^0$, and Cauchy-Schwarz's inequality, the first
 term $\mathcal{S}^{(1)}$ can be estimated as follows:
\begin{align}\label{ht8}
  \big| \mathcal{S}^{(1)} \big|  \leq  C(\|\mathcal{E}^\epsilon(t)\|_{s}^2 +1)
   ( \|\partial^\alpha_x \Theta^\epsilon\|^2 +\|\partial^\alpha_x\U^\epsilon\|^2).
\end{align}

For the  terms $\mathcal{S}^{(i)}\,(i=2,3,4,5)$, we utilize the regularity of $ \rho^0$, $\u^0$ and $\H^0$,
the positivity of $ N^\epsilon+\rho^0$ , Cauchy-Schwarz's inequality and \eqref{ho} to deduce that
\begin{align}\label{ht88}
   \big| \mathcal{S}^{(2)} \big|+ \big|\mathcal{S}^{(3)} \big| + \big|\mathcal{S}^{(4)} \big|+ \big|\mathcal{S}^{(5)} \big|
   \leq  C(\|\mathcal{E}^\epsilon(t)\|_{s}^{2s} +\|\mathcal{E}^\epsilon(t)\|_{s}^{2})+
   C\|\partial^\alpha_x\Theta^\epsilon\|^2,
\end{align}
while for the sixth term $\mathcal{S}^{(6)}$, we integrate by parts, and use Cauchy-Schwarz's inequality and the positivity of
$\Theta^\epsilon+\rho^0$ to obtain that
\begin{align}\label{ht9}
\mathcal{S}^{(6)}=\ &- \left\langle\partial_x^{\alpha-\alpha_1}\left\{\frac{ 2\mu}{N^\epsilon+\rho^0} |\mathbb{D}(\U^\epsilon)|^2
+\frac{\lambda}{N^\epsilon+\rho^0}|\mbox{tr}\mathbb{D}(\U^\epsilon)|^2\right\},\partial_x^{\alpha-\alpha_1}\Theta^\epsilon\right\rangle\nonumber\\
 \leq\ &\eta_{11} \|\nabla \pa \Theta^\epsilon\|^2+C_{\eta_{11}}(\|\mathcal{E}^\epsilon(t)\|_{s}^4+\|\mathcal{E}^\epsilon(t)\|_{s}^{2(s+1)})
\end{align}
for any $\eta_{11}>0$, where $\alpha_1=(1,0,0)$ or $(0,1,0)$ or $(0,0,1)$.  Similarly, we have
\begin{align}\label{ht11}
\big|\mathcal{S}^{(7)}\big| +\big|\mathcal{S}^{(8)}\big|
 \leq  \eta_{12} \|\nabla \pa \Theta^\epsilon\|^2+C_{\eta_{12}}(\|\mathcal{E}^\epsilon(t)\|_{s}^4+\|\mathcal{E}^\epsilon(t)\|_{s}^{2(s+1)})
\end{align}
for any $\eta_{12}>0$.

For the ninth term $\mathcal{S}^{(9)}$, we rewrite it as
\begin{align*}
\mathcal{S}^{(9)} &=  \left\langle\pa\left\{ \frac{1}{N^\epsilon+\rho^0}|\F^\epsilon
+\U^\epsilon\times \G^\epsilon|^2\right\},\pa\Theta^\epsilon\right\rangle\nonumber\\
 &=  \left\langle\pa\left\{ \frac{1}{N^\epsilon+\rho^0}|\F^\epsilon|^2\right\},\pa\Theta^\epsilon\right\rangle\nonumber\\
  &\ \ \ \ +  \left\langle\pa\left\{ \frac{2}{N^\epsilon+\rho^0}\F^\epsilon
  \cdot(\U^\epsilon\times \G^\epsilon)\right\},\pa\Theta^\epsilon\right\rangle\nonumber\\
   &\ \ \ \  +  \left\langle\pa\left\{ \frac{1}{N^\epsilon+\rho^0} |\U^\epsilon\times
   \G^\epsilon|^2\right\},\pa\Theta^\epsilon\right\rangle\nonumber\\
    &:= \mathcal{S}^{(9_1)}+\mathcal{S}^{(9_2)}+\mathcal{S}^{(9_3)}.
      \end{align*}
By Cauchy-Schwarz's inequality and Sobolev's embedding, we can bound the term $\mathcal{S}^{(9_1)}$ by
\begin{align}\label{ht111}
 \mathcal{S}^{(9_1)} =& \left\langle \frac{1}{N^\epsilon+\rho^0}\pa\left(|\F^\epsilon|^2\right),\pa\Theta^\epsilon\right\rangle\nonumber\\
 & +\sum_{\beta \leq \alpha,|\beta|<|\alpha|}
 \left\langle\partial_x^{\alpha-\beta}\left( \frac{1}{N^\epsilon+\rho^0}\right)\partial_x^{\beta}|\F^\epsilon|^2,\pa\Theta^\epsilon\right\rangle\nonumber\\
 \leq & \gamma_1 \|\F^\epsilon\|^4_s+C_{\gamma_1}\|\pa\Theta^\epsilon\|^2(1+\|\mathcal{E}(t)\|^{2(s+1)}_s)
\end{align}
for any $\gamma_1>0$. For  the term $\mathcal{S}^{(9_2)}$, similar to $\mathcal{R}^{(6)}$,  we have
\begin{align}\label{ht12}
   \mathcal{S}^{(9_2)}& = 2\left\langle\partial^\alpha_x \F^\epsilon\cdot
\frac{\U^\epsilon\times\G^\epsilon }{N^\epsilon+\rho^0},
   \partial^\alpha_x\Theta^\epsilon \right\rangle
+  2\big\langle \mathcal{H}^{(11)},
\partial^\alpha_x\Theta^\epsilon
\big\rangle \nonumber\\
 & \leq \frac{1}{16} \| \partial^\alpha_x \F^\epsilon\| ^2+
 C\| \mathcal{E}^\epsilon(t)\| _{s}^2
 \| \partial^\alpha_x\U^\epsilon\| ^2+2 \big\langle \mathcal{H}^{(11)},  \partial^\alpha_x\Theta^\epsilon
\big\rangle,
\end{align}
where
\begin{align*}
 \mathcal{H}^{(11)} =\partial^\alpha_x\left\{ \F^\epsilon\cdot
\frac{\U^\epsilon\times\G^\epsilon }{N^\epsilon+\rho^0}\right\}-\partial^\alpha_x \F^\epsilon\cdot
\frac{\U^\epsilon\times\G^\epsilon }{N^\epsilon+\rho^0}.
\end{align*}
By the Cauchy-Schwarz's and Moser-type inequalities, we obtain that
 \begin{align}\label{ht13}
 &  \ \ \  \ 2\big|\big\langle\mathcal{H}^{(11)},
  \partial^\alpha_x\Theta^\epsilon\big\rangle\big| \nonumber\\
  &\leq  2\big\| \mathcal{H}^{(11)}\big\| \cdot \| \partial^\alpha_x\Theta^\epsilon \|  \nonumber\\
  & \leq C \left[\left\Vert D_x^1\left(\frac{\U^\epsilon\times\G^\epsilon }{N^\epsilon+\rho^0}\right)\right\Vert_{L^\infty}\| \F^\epsilon\| _{s-1}
   +\| \F^\epsilon\| _{L^\infty}\left\Vert\frac{\U^\epsilon\times\G^\epsilon }{N^\epsilon+\rho^0}\right\Vert_{s}\right]\| \partial^\alpha_x\Theta^\epsilon \| \nonumber\\
   &\leq  \gamma_2 \| \F^\epsilon\| ^2_{s-1}+C_{\gamma_2} (\| \mathcal{E}^\epsilon(t)\| ^2_{s}+1) \| \partial^\alpha_x
   \Theta^\epsilon\| ^2.
 \end{align}
for any $\gamma_2>0$. The term $\mathcal{S}^{(9_3)}$ can be bounded as follows, using the Cauchy-Schwarz and Moser-type inequalities.
\begin{align}\label{ht133}
\big|\mathcal{S}^{(9_3)} \big|  \leq  C\|\pa\Theta^\epsilon\|^2(1+\|\mathcal{E}(t)\|^{2(s+1)}_s).
\end{align}

 By  the regularity of $ \theta^0$, $\u^0$ and $\H^0$, the positivity of $ \Theta^\epsilon+\rho^0$,
 and Cauchy-Schwarz's inequality, the first
 terms $\mathcal{S}^{(10)}$ and  $\mathcal{S}^{(13)}$   can be bounded as follows:
\begin{align}\label{ht14}
   \big| \mathcal{S}^{(10)} \big| +\big| \mathcal{S}^{(13)} \big|  \leq  C(\|\mathcal{E}^\epsilon(t)\|_{s}^{2s} +1)
   ( \|\partial^\alpha_x \Theta^\epsilon\|^2 +\|\partial^\alpha_x\U^\epsilon\|^2+\|\partial^\alpha_x\G^\epsilon\|^2).
\end{align}
In a manner similar to $\mathcal{S}^{(9_2)}$, we can control the term $ \mathcal{S}^{(11)}$ by
\begin{align}\label{ht15}
  \big|\mathcal{S}^{(11)}\big|\leq   \gamma_3 \| \F^\epsilon\| ^2_{s-1}+C_{\gamma_3} (\| \mathcal{E}^\epsilon(t)\| ^{2s}_{s}+1) ( \|\partial^\alpha_x \Theta^\epsilon\|^2 +\|\partial^\alpha_x\U^\epsilon\|^2+\|\partial^\alpha_x\G^\epsilon\|^2)
\end{align}
 for any  $\gamma_3>0$.
Finally, similarly to $\mathcal{S}^{(9_3)}$, the term $ \mathcal{S}^{(12)}$ can be bounded  by
\begin{align}\label{ht16}
 \big| \mathcal{S}^{(12)}\big|   \leq & C\|\pa\Theta^\epsilon\|^2(1+\|\mathcal{E}(t)\|^{2(s+1)}_s).
\end{align}

Substituting  \eqref{ht1}--\eqref{ht16} into \eqref{HT1}, we conclude that
 \begin{align}\label{HT202}
& \frac12\frac{\dif}{\dif t} \langle \partial^\alpha_x\Theta^\epsilon,
\partial^\alpha_x\Theta^\epsilon \rangle +
\kappa \left\langle  \frac{1}{N^\epsilon+\rho^0}\nabla\partial^\alpha_x \Theta^\epsilon , \nabla\partial^\alpha_x\Theta^\epsilon\right\rangle \nonumber\\
& -(\eta_9+  \eta_{10}+\eta_{11}+\eta_{12})\| \nabla
\partial^\alpha_x\Theta^\epsilon \| ^2  \nonumber\\
\leq  \ &  {C}_{{\eta,\gamma}}\big[(\| \mathcal{E}^\epsilon(t)\| ^{2(s+1)}_{s}
    +\| \mathcal{E}^\epsilon(t)\| ^2_{s}+\| \mathcal{E}^\epsilon(t)\| _{s}+1) \| \mathcal{E}^\epsilon(t)\| ^2_{s} \big]\nonumber\\
  & +\gamma_1 \| \F^\epsilon\| ^4_{s} +\left(\gamma_2+\gamma_3+\frac{1}{16}\right)\| \F^\epsilon\| ^2_{s}
\end{align}
for some constant $ {C}_{{\eta,\gamma}}>0$ depending on $\eta_i$ ($i=9,10,11,12$) and $\gamma_j$ ($j=1,2,3$).

Applying the operator $\partial^\alpha_x$ to  \eqref{error3} and \eqref{error4}, multiplying  the results
    by ${ }\partial^\alpha_x\F^\epsilon$ and   $\partial^\alpha_x\G^\epsilon$  respectively, and integrating then over
    ${\mathbb T}^3$, one obtains that
\begin{align}
  &  \frac12\frac{\dif}{\dif t}({\epsilon }\|\partial^\alpha_x\F^\epsilon\| ^2
  +\| \partial^\alpha_x\G^\epsilon\| ^2) +  \| \partial^\alpha_x\F^\epsilon\| ^2\nonumber\\
 & +\int (\cu \partial^\alpha_x\F^\epsilon\cdot \partial^\alpha_x \G^\epsilon-\cu \partial^\alpha_x\G^\epsilon\cdot
 \partial^\alpha_x\F^\epsilon)\dif x\nonumber\\
 = \ &   \left\langle [\partial^\alpha_x(\U^\epsilon\times \H^0)+\partial^\alpha_x(\u^0\times \G^\epsilon)]
  - \partial^\alpha_x(\U^\epsilon\times \G^\epsilon),  \partial^\alpha_x\F^\epsilon\right\rangle\nonumber\\
 & -  \left\langle {\epsilon} \partial^\alpha_x\partial_t
 \cu \H^0+\epsilon\partial^\alpha_x\partial_t(\u^0\times \H^0),
 \partial^\alpha_x\F^\epsilon\right\rangle. \label{L2MM}
\end{align}

Following a process similar to that in \cite{JL} and applying \eqref{L2MM}, we finally obtain that
\begin{align}
 &\frac12  \frac{\dif}{\dif t}({\epsilon }\|\partial^\alpha_x\F^\epsilon\|^2+\| \partial^\alpha_x\G^\epsilon\| ^2)
   +\frac{3}{4}\| \partial^\alpha_x\F^\epsilon\| ^2  \nonumber\\
  &\qquad \qquad \leq C(\| \mathcal{E}^\epsilon(t)\|_{s}^2+1)\|
  (\partial^\alpha_x\U^\epsilon,\partial^\alpha_x\G^\epsilon)\|^2+C\epsilon^2. \label{L2M2b}
\end{align}

Combining \eqref{HN1},  \eqref{HU202}, and  \eqref{HT202} with \eqref{L2M2b}, summing up $\alpha$ with
$0\leq|\alpha|\leq s$, using the fact that $N^\epsilon +\rho^0\geq \hat N+\hat \rho>0$, $\F^\epsilon \in C^l([0,T],H^{s-2l})$
($l=0,1$), and choosing $\eta_i$ ($i=1,\dots,12$) and $\gamma_1,\gamma_2,\gamma_3$
sufficiently small, we obtain \eqref{H2}. This completes the proof of Lemma 3.2.
\end{proof}

With the estimate  \eqref{H2} in hand, we can now  prove Proposition \ref{P31}.

\begin{proof}[Proof of Proposition \ref{P31}]
As in \cite{PW,JL}, we introduce an $\epsilon$-weighted energy functional
$$\Gamma^\epsilon(t)  =  \v \mathcal{E}^\epsilon(t)\v ^2_{s}.$$
 Then, it follows from \eqref{H2} that there exists a constant
 $\epsilon> 0$ depending only on $T$, such that
for any $\epsilon\in (0,\epsilon]$ and any  $t\in (0,T]$,
\begin{align}\label{gma}
\Gamma^\epsilon(t)\leq C\Gamma^\epsilon(t = 0)+C\int^t_0\Big\{
\big((\Gamma^\epsilon)^{s}+
 \Gamma^\epsilon +1\big)\Gamma^\epsilon\Big\}(\tau)\dif \tau
+ C\epsilon^2.
\end{align}
 Thus, applying the Gronwall lemma to \eqref{gma}, and keeping in mind that $\Gamma^\epsilon \,(t=0)\leq C\epsilon^2$
 and Proposition \ref{P31}, we find that there exist a $0<T_1<1$ and an $\epsilon>0$, such that $T^\epsilon\geq T_1$
for all $\epsilon\in (0,\epsilon]$ and
$\Gamma^\epsilon(t)\leq C\epsilon^2$ for all $ t \in (0,T_1]$. Therefore, the desired a priori estimate
\eqref{www} holds. Moreover, by the standard continuation induction argument, we can extend
$T^\epsilon\geq T_0$  to any $T_0<T_*$.
\end{proof}


\medskip \noindent
{\bf Acknowledgements:}
The authors are very grateful to the referees for their  helpful suggestions, which  improved the earlier
version of this paper.
Jiang is supported by the National Basic Research Program under the Grant 2011CB309705 and NSFC (Grant Nos. 11229101, 11371065);
and Li is supported  by NSFC (Grant No. 11271184), PAPD, and  NCET-11-0227.

\bibliographystyle{plain}

\end{document}